\documentclass[a4paper,12pt,english]{article}
\usepackage{amsmath}
\usepackage{amssymb}
\usepackage{latexsym}
\usepackage{amsthm}
\usepackage{url}
\usepackage{multirow}
\usepackage{tabularx,booktabs,adjustbox}
\usepackage{enumitem}
\usepackage{verbatim}
\usepackage{authblk}
\numberwithin{equation}{section}

\usepackage[x11names]{xcolor} 
\colorlet{heatcol}{Red1!70!} 
\colorlet{wavecol}{DodgerBlue1!70!} 
\colorlet{mixedcol}{Maroon2!70!} 

\usepackage{amsthm}
\usepackage{tikz,pgfplots}
\usetikzlibrary{patterns}
\usepackage{float}
\usepackage{caption}
\usepackage{subfig}

\def\d{\displaystyle}
\def\p{\partial}

\def\e{\varepsilon}
\def\k{\kappa}

\def\l{\lambda}
\def\n{\nu}

\def\R{\mathbb{R}}
\def\N{\mathbb{N}}

\def\B{{\mathcal B}}
\def\F{{\mathcal F}}

\def\ints{\int_{\R^n}}
\def\o{\overline}
\def\til{~}

\DeclareMathOperator*{\supp}{supp}
\DeclareMathOperator*{\sgn}{sgn}

\newtheorem{theorem}{Theorem}

\newtheorem{remark}{Remark}[section]
\newtheorem{lemma}{Lemma}

\newtheorem{definition}{Definition}
\title{Heat-like and wave-like lifespan estimates\\
	for solutions of semilinear damped wave equations\\
	via a Kato's type lemma}
\author{
	Ning-An Lai,
	\quad
	Nico Michele Schiavone,
	\quad
	Hiroyuki Takamura
	}

\date{}

\AtEndDocument{%
	\par
	\medskip
	\begin{tabular}{@{}l@{}}%
	\end{tabular}}

\AtEndDocument{%
	\par
	\medskip
	\begin{tabular}{@{}l@{}}%
		\textsc{Ning-An Lai}
		\\
		Institute of Nonlinear Analysis and Department of Mathematics,
		Lishui University\\
		No. 1 Xueyuan Rd., Lishui City 323000, Zhejiang,
		China
		\\
		\textit{E-mail}: \texttt{ninganlai@lsu.edu.cn}
	\end{tabular}}

\AtEndDocument{%
	\par
	\medskip
	\begin{tabular}{@{}l@{}}%
		\textsc{Nico Michele Schiavone}
		\\
		Department of Mathematics \lq\lq Guido Castelnuovo'',
		Sapienza University of Rome
		\\
		Piazzale Aldo Moro 5, Roma 00185, Italy
		\\
		\textit{E-mail}: \texttt{schiavone@mat.uniroma1.it}
	\end{tabular}}

\AtEndDocument{%
	\par
	\medskip
	\begin{tabular}{@{}l@{}}%
		\textsc{Hiroyuki Takamura}
		\\
		Mathematical Institute,
		Tohoku University
		\\
		Aoba, Sendai 980-8578, Japan
		\\
		\textit{E-mail}: \texttt{hiroyuki.takamura.a1@tohoku.ac.jp}
\end{tabular}}

\begin{document}
	\maketitle
	
		\begin{center}\footnotesize
			\begin{tabularx}{0.7\textwidth}{lX}
				{\bf Keywords:}
				& 
				semilinear wave equation, scale-invariant damping, blow-up, lifespan
				\\[0.35cm]
				{\bf MSC2020:}
				&
				primary 35L71, secondary 35B44
			\end{tabularx}
			%
		\end{center}

	\smallskip
	
	\begin{abstract}
    In this paper we study several semilinear damped wave equations with \lq\lq subcritical" nonlinearities, focusing on demonstrating lifespan estimates for energy solutions.
	Our main concern is on equations with scale-invariant damping and mass.
	Under different assumptions imposed on the initial data, lifespan estimates from above are clearly showed. The key fact is that we find \lq\lq transition surfaces", which distinguish lifespan estimates between \lq\lq wave-like" and \lq\lq heat-like" behaviours.
	Moreover we conjecture that the lifespan estimates on the
	\lq\lq transition surfaces" can be logarithmically improved.
	As direct consequences, we reorganize the blow-up results and lifespan estimates for the massless case in which the \lq\lq transition surfaces" degenerate to \lq\lq transition curves". Furthermore, we obtain improved lifespan estimates in one space dimension, comparing to the known results.

      We also study semilinear wave equations with the scattering damping and negative mass term, and find that if the decay rate of the mass term equals to $2$, the lifespan estimate is the same as one special case of the equations with the scale-invariant damping and positive mass.

      The main strategy of the proof consists of a Kato's type lemma in integral form, which is established by iteration argument.
		
	\end{abstract}
	
	\tableofcontents


\section{Introduction}
The aim of the present work is to study blow-up phenomena and lifespan estimates for solutions
of Cauchy problem with small data for several semilinear damped wave models. Indeed we mainly concern about semilinear wave equations with the scale-invariant damping, mass and power-nonlinearity
\begin{equation}
\label{eq:main_problem}
\left\{
\begin{aligned}
& u_{tt} - \Delta u + \frac{\mu_1}{1+t} u_t + \frac{\mu_2}{(1+t)^{2}} u = |u|^p, \quad\text{in $\R^n\times (0,T)$}, \\
& u(x,0)=\e f(x), \quad u_t(x,0)=\e g(x), \quad x\in\R^n,
\end{aligned}
\right.
\end{equation}
where  $\mu_1, \mu_2 \in \R$, $p>1$, $n\in\N$, $T>0$ and $\e>0$ is a \lq\lq small'' parameter. In particular, we are interested in exploring a competition between so-called \lq\lq heat-like'' and \lq\lq wave-like'' behaviour of the solutions, which concerns not only critical exponents, but also lifespan estimates, in a way that we will clarify later.

Let us firstly denote energy and weak solutions of our problem \eqref{eq:main_problem}.

\begin{definition}
	We say that $u$ is an energy solution of \eqref{eq:main_problem} over $[0,T)$ if
	\begin{equation*}
		u \in C([0,T), H^1(\R^n)) \cap C^1([0,T),L^2(\R^n)) \cap C((0,T), L^p_{loc}(\R^n))
	\end{equation*}
	satisfies $u(x,0)=\e f(x)$ in $H^1(\R^n)$, $u_t(x,0) = \e g(x)$ in $L^2(\R^n)$ and
	\begin{equation}
		\label{eq:energy_solution}
		\begin{split}
			&\int_{\R^n}u_t(x,t)\phi(x,t)dx-\int_{\R^n}\e g(x)\phi(x,0)dx\\
			&+\int_0^tds\int_{\R^n}\left\{-u_t(x,s)\phi_t(x,s)+\nabla u(x,s)\cdot\nabla\phi(x,s)\right\}dx \\
			&+\int_0^tds\int_{\R^n}\frac{\mu_1 }{1+s}u_t(x,s) \phi(x,s)dx + \int_0^t ds \ints \frac{\mu_2}{(1+s)^{2}}u(x,s)\phi(x,s) \\
			= & \int_0^tds\int_{\R^n}|u(x,s)|^p\phi(x,s)dx
		\end{split}
	\end{equation}
	for $t\in[0,T)$ and any test function $\phi \in C_0^\infty(\R^n \times [0,T))$.
	
	Employing the integration by parts in the above equality and letting $t\to T$, we reach to the definition of the weak solution of \eqref{eq:main_problem}, that is
	\begin{equation*}
		\begin{split}
			&\int_{\R^n\times[0,T)}
			u(x,s)
			\left\{\phi_{tt}(x,s)-\Delta \phi(x,s)
			-\frac{\p}{\p s} \left(\frac{\mu_1}{1+s}\phi(x,s)\right)
			+ \frac{\mu_2}{(1+s)^{2}} \phi(x,s) \right\}dxds\\
			=&\ 
			\e \int_{\R^n} \left\{ \mu_1 f(x)\phi(x,0) + g(x)\phi(x,0) - f(x)\phi_t(x,0) \right\} dx
			+\int_{\R^n\times[0,T)}|u(x,s)|^p\phi(x,s)dxds.
		\end{split}
	\end{equation*}
		
\end{definition}

We recall that the \emph{critical exponent} $p_{crit}$ of \eqref{eq:main_problem} is the smallest exponent $p_{crit}>1$ such that, if $p>p_{crit}$, there exists a unique global energy solution to the problem, whereas if $1<p \le p_{crit}$ the solution blows up in finite time. In the latter case, one is also interested in finding estimates for the \emph{lifespan} $T_\e$, which is the maximal existence time of the solution, depending on the parameter $\e$.

Our principal model is the one in \eqref{eq:main_problem}, for which we obtain Theorem\til\ref{thm1} and Theorem\til\ref{thm:H=0}, according to the different conditions imposed on the initial data.
As straightforward consequences, we also obtain Theorem\til\ref{cor1} and Theorem\til\ref{cor2} for the massless case, i.e. the model with $\mu_2=0$. The lifespan estimate in dimension $n=1$ in this case is improved, comparing to the known results. Moreover, we continue the study of semilinear wave equations with scattering damping, negative mass term and power nonlinearity, introduced by the authors in \cite{LST19.1, LST19.2}.

The paper is organized in this way: in the rest of the Introduction, we will sketch the background of the problems under consideration and we will exhibit our results, which will be proved in Section\til\ref{sec:proof}, exploiting, as main tool, a Kato's type lemma in integral form presented in Section\til\ref{sec:kato}.

\subsection{Heat vs. wave}\label{sub:HvsW}
Let us start considering the toy-models of the wave and heat equations:
\begin{equation*}
\label{wave-heat}
\left\{
\begin{aligned}
& u_{t} - \Delta u = |u|^p, 
\\
& u(x,0)=\e f(x), 
\end{aligned}
\right.
\qquad\qquad
\left\{
\begin{aligned}
& u_{tt} - \Delta u = |u|^p, 
\\
& (u,u_t)(x,0)=\e (f,g)(x). 
\end{aligned}
\right.
\end{equation*}
Nowadays the study of these two equations is almost classic: the well-known results include the lifespan estimates and the critical exponents, which are the so-called Fujita exponent $p_F(n)$ and the Strauss exponent $p_S(n)$, corresponding to the heat and the wave equation respectively. For the purpose of this work, let us define these two exponents for all $\nu\in\R$:
\begin{equation*}\label{def:fujita-strauss}
p_F(\n) :=
\left\{
\begin{aligned}
&1 + \frac{2}{\n} & &\text{if $\n>0$}, \\
&+\infty & &\text{if $ \n \le 0$,}
\end{aligned}
\right.
\qquad
p_S(\n) :=
\left\{
\begin{aligned}
&\frac{\n+1+\sqrt{\n^2+10\n-7}}{2(\n-1)} & &\text{if $\n > 1$}, \\
&+\infty & &\text{if $ \n \le 1$.}
\end{aligned}
\right.
\end{equation*}
We remark that
\begin{align*}
1 < p < p_F(\n)
&\Longrightarrow
\gamma_F(p,\n):=2-\n(p-1) > 0,
\\
1 < p < p_S(\n)
&\Longrightarrow
\gamma_S(p,\n):= 2+(\n+1)p-(\n-1)p^2 > 0.
\end{align*}
In particular, if $\n>0$, $p_F(\n)$ is the solution of the linear equation $\gamma_F(p,\n)=0$, whereas if $\n > 1$, $p_S(\n)$ is the positive solution of the quadratic equation $\gamma_S(p,\n)=0$.
Although the expression $\gamma_S(p,\nu)$ is well-known in the literature, the introduction of $\gamma_F(p,\n)$ is justifyed from the fact that $\gamma_F$ plays for the heat equation the same role that $\gamma_S$ plays for the wave equation, as it emerge from the lifespan estimates.

Suppose for the simplicity that $f,g>0$ are compactly supported (for different conditions on the initial data, we can have different lifespan estimates, see Subsection\til\ref{subsec:h=0}). We have that the blow-up results are the ones collected in the following table.

\begin{center}\footnotesize
	\begin{adjustbox}{width=\textwidth}
	\begin{tabular}{lllll}
		\toprule
		&& {Heat} && {Wave} \\
		\midrule[\heavyrulewidth]
		
		\begin{tabular}
			{@{}l@{}}
			Critical exponent $p_{crit}$
		\end{tabular}
		&&
		$p_F(n)$
		&&
		$p_S(n)$
		
		\\\midrule
		
		\begin{tabular}
			{@{}l@{}}
			Subcritical lifespan $T_\e$
			\\
			for $1<p<p_{crit}$
		\end{tabular}
		&&
		$\sim \e^{-2(p-1)/\gamma_F(p,n)}$
		&&
		\begin{tabular}
			{@{}l@{}}
			$\sim\e^{-(p-1)/\gamma_F(p,n-1)}$
			\\ \quad if $n=1$ or $n=2$, $1<p<2$
			\vspace{2mm}
			\\
			$\sim a(\e)$
			\\ \quad if $n=p=2$, $\e^2 a^2 \log(1+a)=1$
			\vspace{2mm}
			\\
			$\sim \e^{-2p(p-1)/\gamma_S(p,n)}$
			\\ \quad if $n=2,2<p<p_S(n)$ or $n\ge3$
		\end{tabular}
	
		\\\midrule
		
		\begin{tabular}
			{@{}l@{}}
			Critical lifespan $T_\e$
			\\
			for $p=p_{crit}$
		\end{tabular}
		&&
		$\sim \exp(C \e^{-(p-1)})$
		&&
		\begin{tabular}
			{@{}l@{}}
			$\sim \exp(C \e^{-p(p-1)})$
			\\
			\quad (the lower bound is open for $n\ge9$ in general)
		\end{tabular}
		
		\\
		\bottomrule
	\end{tabular}
	\end{adjustbox}
\end{center}
Here and in the following, we use the notation $F \lesssim G$ (respectively $F \gtrsim G$) if there exists a constant $C>0$ independent of $\e$ such that $F \le C G$ (respectively $F \ge C G$), and the notation $F \sim G$ if $F\lesssim G$ and $F \gtrsim G$.

For a more detailed story of these results, we refer to the book \cite{ER}, the doctoral thesis \cite{wak14-thesis}, the introductions of \cite{IKTW19-aa,T15,TW11,TW14} and the references therein.

For the comparison between the heat and wave equations, let us introduce an informal but evocative notation to describe the behaviour of the critical exponent and of the lifespan estimates in our models. We will call the critical exponent
\emph{heat-like} if it is related to the Fujita exponent, i.e. $p_{crit} = p_F(\nu)$ for some $\nu\in\R$,
whereas we will call it \emph{wave-like} if it is related to the Strauss exponent, i.e. $p_{crit} = p_S(\nu)$ for some $\nu\in\R$.

Similarly, we will say that the lifespan estimate is \emph{heat-like} if it is related in some way to the one of the heat equation, i.e. to the exponent $2(p-1)/\gamma_F(p,\nu)$ in the subcritical case and to $\exp(\e^{-(p-1)})$ in the critical one, whereas we will say it \emph{wave-like} if related to the one of the wave equation, i.e. to the exponent $2p(p-1)/\gamma_S(p,\nu)$ in the subcritical case and to $\exp(\e^{-p(p-1)})$ in the critical one. However, we also define a \emph{mixed-type} behaviour when the lifespan estimate is related to $2p(p-1)/\gamma_F(p,\nu)$ in the subcritical case (as we will see in Theorem\til\ref{cor2} \& \ref{thm:H=0}), to remark that the lifespan is larger respect to the heat-like one, due to the additional $p$ in the exponent.

\subsection{Damped wave equation}

Let us proceed further by adding the damping term $\mu/(1+t)^{\beta}$ to the wave equation, with $\mu\ge 0 $ and $\beta\in\R$, hence we consider the Cauchy problem
\begin{equation}
\label{eq:damped}
\left\{
\begin{aligned}
& u_{tt} - \Delta u + \frac{\mu}{(1+t)^\beta} u_t = |u|^p, \quad\text{in $\R^n\times (0,T)$}, \\
& u(x,0)=\e f(x), \quad u_t(x,0)=\e g(x), \quad x\in\R^n.
\end{aligned}
\right.
\end{equation}
According to the works by Wirth \cite{Wir04, Wir06, Wir07}, in the study of the associated homogeneous problem
\begin{equation}
\label{eq:damped0}
\left\{
\begin{aligned}
& u^0_{tt}-\Delta u^0+\frac{\mu}{(1+t)^\beta}u^0_t=0, \\
& u^0(x,0)=f(x), \quad u^0_t(x,0)=g(x), 
\end{aligned}
\right.
\end{equation}
we can classify the damping term accordingly to the different values of $\beta$ into four cases.
When $\beta<1$, the damping term is said to be \emph{overdamping} and the solution
does not decay to zero when $t\to\infty$.
If $-1 \le \beta<1$, the solution behaves
like that of the heat equation and we say that the damping term is \emph{effective}. Hence, the term $u^0_{tt}$ in \eqref{eq:damped0} has
no influence on the behavior of the solution and the $L^p-L^q$ decay estimates of the solution are almost the same as those of the heat equation.
In contrast, when $\beta>1$, it is known that the solution behaves like that
of the wave equation, which means that the damping term in \eqref{eq:damped0} has no
influence on the behavior of the solution. In fact, in this case the solution
scatters to that of the free wave equation when $t\to\infty$, and thus we say that
we have \emph{scattering}.
Finally, when $\beta=1$, the equation in \eqref{eq:damped0} is invariant under the scaling
\begin{equation*}
	\widetilde{u^0}(x,t) := u^0(\sigma x, \sigma(1+t)-1), \quad \sigma>0,
\end{equation*}
and hence we say that the damping term is \emph{scale-invariant}. In this case the behaviour of the solution of \eqref{eq:damped0} has been observed to be determined by the value of $\mu$.
We summarize all the classifications of the damping term in
\eqref{eq:damped0} in the next table.
\begin{center}
	\begin{tabular}{lll}
		\toprule
		Range of $\beta$ && Classification  \\
		\midrule
		
		$\beta\in(-\infty,-1)$ && overdamping\\
		
		$\beta\in[-1,1)$ && effective\\
		
		$\beta=1$ && scaling invariant \\
		
		$\beta\in(1,\infty)$ && scattering \\
		\bottomrule
	\end{tabular}
\end{center}

Let us return to problem \eqref{eq:damped}, which inherits the above terminology and has very different behaviours from case to case.
Indeed, in the overdamping case the solution exist globally for any $p>1$. In the effective case, the problem is heat-like, both in the critical exponent and in the lifespan estimates,
while in the scattering case the problem seems to be wave-like. Finally, the scale-invariant case has an intermediate behaviour, and a competition between heat-like and wave-like arises. Before moving to the last case, let us collect in the following two tables some global existence and blow-up results for $\beta\neq1$, at the best of our knowledge.

\begin{center}
	\footnotesize
	
	\begin{adjustbox}{width=\textwidth}
		\begin{tabular}{llll}
			\toprule
			\multicolumn{4}{c}{\emph{Global-in-time existence for $\beta\neq1$}}
			\\\midrule
			Authors & Range of $\beta$ & Dimension $n$ & Exponent $p$
			\\\midrule[\heavyrulewidth]
			
			\begin{tabular}
				{@{}l@{}} Ikeda, Wakasugi \cite{IW20}
			\end{tabular}
			&
			$\beta<-1$
			&
			\begin{tabular}
				{@{}l@{}}
				$n\ge1$
			\end{tabular}
			&
			$p>1$
			
			\\\midrule
			
			\begin{tabular}
				{@{}l@{}}
				Wakasugi \cite{wakasugi17}
			\end{tabular}
			&
			$\beta=-1$
			&
			\begin{tabular}
				{@{}l@{}}
				$n=1,2$
				\\
				$n\ge3$
			\end{tabular}
			&
			\begin{tabular}
				{@{}l@{}}
				$p>p_F(n)$ 
				\\
				$p_F(n)<p<\frac{n}{n-2}$ 
			\end{tabular}
			
			\\\midrule
			
			Todorova, Yordanov \cite{TY01}
			&
			$\beta=0$
			&
			\begin{tabular}
				{@{}l@{}}
				$n=1,2$
				\\
				$n\ge3$
			\end{tabular}
			&
			\begin{tabular}
				{@{}l@{}}
				$p>p_F(n)$ 
				\\
				$p_F(n)<p\le\frac{n}{n-2}$ 
			\end{tabular}

			\\\midrule
			
			\begin{tabular}
				{@{}l@{}}
				D'Abbicco, Lucente, Reissig \cite{DLR15} \\
				Nishihara \cite{Nis11} \\
				Lin, Nishihara, Zhai \cite{LNZ12}
			\end{tabular}
			&
			\begin{tabular}
				{@{}l@{}}
				$-1<\beta<1$
				\\
				$\beta\neq0$
			\end{tabular}
			&
			\begin{tabular}
				{@{}l@{}}
				$n=1,2$
				\\
				$n\ge3$
			\end{tabular}
			&
			\begin{tabular}
				{@{}l@{}}
				$p>p_F(n)$ 
				\\
				$p_F(n)<p<\frac{n+2}{n-2}$ 
			\end{tabular}
			
			\\\midrule
			
			Liu, Wang \cite{LW}
			&
			$\beta>1$
			&
			\begin{tabular}
				{@{}l@{}}
				$n=3,4$
			\end{tabular}
			&
			\begin{tabular}
				{@{}l@{}}
				$p>p_S(n)$ 
			\end{tabular}
			
			\\
			\bottomrule
		\end{tabular}
		
	\end{adjustbox}
\end{center}

\begin{center}
	\footnotesize
	
	\begin{adjustbox}{width=\textwidth}

	\begin{tabular}{llll}
		\toprule
		\multicolumn{4}{c}{\emph{Blow-up in finite time for $\beta\neq1$}}
		\\\midrule
		Authors & Range of $\beta$ & Exponent $p$ & Lifespan $T_\e$
		\\\midrule[\heavyrulewidth]

		\begin{tabular}
			{@{}l@{}}
			Fujiwara,
			Ikeda,
			Wakasugi \cite{FIW19}
			\\
			Ikeda, Inui \cite{II19}
		\end{tabular}
		&
		$\beta=-1$
		&
		\begin{tabular}
			{@{}l@{}}
			$1<p<p_F(n)$
			\\
			$p=p_F(n)$
		\end{tabular}
		&
		\begin{tabular}
			{@{}l@{}}
			$\sim \exp( C \e^{-{2(p-1)}/{\gamma_F(p,n)}})$
			\\
			$\sim \exp\exp(C\e^{-(p-1)})$
		\end{tabular}
		
		\\\midrule
		
		\begin{tabular}
			{@{}l@{}}
			Li, Zhou \cite{LZ95}, Zhang \cite{Zhang01}\\
			Todorova, Yordanov \cite{TY01}\\
			Kirane, Qafsaoui \cite{KQ01}\\
			Ikeda, Ogawa \cite{IO16}, Lai, Zhou \cite{LaiZ19}
			\\
			Ikeda, Wakasugi \cite{IW15}, Nishihara \cite{Nis11}\\
			Fujiwara, Ikeda, Wakasugi \cite{FIW19}
		\end{tabular}
		&
		$\beta=0$
		&
		\begin{tabular}
			{@{}l@{}}
			$1<p<p_F(n)$ \\
			$p=p_F(n)$
		\end{tabular}
		&
		\begin{tabular}
			{@{}l@{}}
			$\sim \e^{-{2(p-1)}/{\gamma_F(p,n)}}$
			\\
			$\sim \exp(C\e^{-(p-1)}) $
		\end{tabular}

		\\\midrule
		
		\begin{tabular}
			{@{}l@{}}
			Fujiwara,
			Ikeda,
			Wakasugi \cite{FIW19}
			\\
			Ikeda, Inui \cite{II19}\\
			Ikeda, Ogawa \cite{IO16}\\
			Ikeda, Wakasugi \cite{IW15}
		\end{tabular}
		&
		\begin{tabular}
			{@{}l@{}}
			$-1<\beta<1$
			\\
			$\beta\neq0$
		\end{tabular}
		&
		\begin{tabular}
			{@{}l@{}}
			$1<p<p_F(n)$
			\\
			$p=p_F(n)$
		\end{tabular}
		&
		\begin{tabular}
			{@{}l@{}}
			$\sim \e^{-\frac{2(p-1)}{(1+\beta) \gamma_F(p,n)}}$
			\\
			$\sim \exp(C\e^{-(p-1)}) $
		\end{tabular}
		
		\\\midrule
		
		\begin{tabular}
			{@{}l@{}}
			Lai, Takamura \cite{LT18}
			\\
			Wakasa, Yordanov \cite{WY19}
		\end{tabular}
		&
		$\beta>1$
		&
		\begin{tabular}
			{@{}l@{}}
			$1<p<p_S(n)$\\
			$p=p_S(n)$
		\end{tabular}
		&
		\begin{tabular}
			{@{}l@{}}
			$\lesssim \e^{-2p(p-1)/\gamma_S(p,n)}$
			\\
			$\lesssim \exp(C\e^{-p(p-1)})$
		\end{tabular}

		\\
		\bottomrule
	\end{tabular}
	
\end{adjustbox}
\end{center}


\subsection{Scale-invariant damped wave equation}
\label{sub:damped}

We consider now \eqref{eq:damped} for $\beta=1$, hence we consider the Cauchy problem
\begin{equation}
\label{eq:damped_b=1}
\left\{
\begin{aligned}
& u_{tt} - \Delta u + \frac{\mu}{1+t} u_t = |u|^p, \quad\text{in $\R^n\times (0,T)$}, \\
& u(x,0)=\e f(x), \quad u_t(x,0)=\e g(x), \quad x\in\R^n.
\end{aligned}
\right.
\end{equation}
The scale-invariant problem has been studied intensively in the last years. This great interest is motivated by the fact that, differently from the damped wave equation with $\beta\neq1$, in the scale-invariant case the results depend also on the damping coefficient $\mu$, for determining both the critical exponent and the lifespan estimate. Hence, the situation is a bit more complicated, being the scale-invariant case the threshold between the effective ($-1 \le \beta<1$) and non-effective ($\beta>1$) damping, and hence the threshold between a heat-like and a wave-like behaviour.

In the following two tables we collect, at the best of our knowledge, results concerning the existence and the blow-up for the scale-invariant damping.

\begin{center}
	\footnotesize
	
	\begin{adjustbox}{width=\textwidth}
		
		\begin{tabular}{llll}
			\toprule
			\multicolumn{4}{c}{\emph{Global-in-time existence for $\beta=1$}}
			\\\midrule
			Authors & Dimension $n$ & Coefficient $\mu$ & Exponent $p$
			\\\midrule[\heavyrulewidth]

			D'Abbicco \cite{Dab15}
			& \begin{tabular}
				{@{}l@{}}
				$n=1$ \\ $n=2$ \\ $n\ge3$
			\end{tabular}
			& \begin{tabular}
				{@{}l@{}}
				$\mu \ge \tfrac{5}{3}$ \\ $\mu \ge 3$ \\ $\mu\ge n+2$
			\end{tabular}
			& \begin{tabular}
				{@{}l@{}}
				$p> p_F(1)$ \\ $p>p_F(2)$ \\ $p_F(n)<p \le \tfrac{n}{n-2}$
			\end{tabular}
			
			\\\midrule
			
			\begin{tabular}
				{@{}l@{}}
				D'Abbicco, Lucente, Reissig \cite{DLR15}\\
				Kato, Sakuraba \cite{KS19}, Lai \cite{LaiN}
			\end{tabular}
			&
			$n=2,3$
			&
			\begin{tabular}
				{@{}l@{}}
				$\mu=2$
			\end{tabular}
			& \begin{tabular}
				{@{}l@{}}
				$p> p_S(n+2)$
			\end{tabular}
			
			\\\midrule
			
			\begin{tabular}
				{@{}l@{}}
				\begin{tabular}
					{@{}l@{}}
					D'Abbicco, Lucente \cite{DL15}
				\end{tabular}
			\end{tabular}
			& \begin{tabular}
				{@{}l@{}}
				$n\ge5$\\
				(odd dim., rad. symm.)
			\end{tabular}
			& \begin{tabular}
				{@{}l@{}}
				$\mu =2$
			\end{tabular}
			&
			\begin{tabular}
				{@{}l@{}}
				$p_S(n+2) < p < \min\left\{2,\tfrac{n+1}{n-3}\right\}$
			\end{tabular}
			
			\\\midrule
			
			\begin{tabular}
				{@{}l@{}}
				Palmieri \cite{Pal19-mmas}
			\end{tabular}
			&
			$n\ge 4$ (even dim.)
			& \begin{tabular}
				{@{}l@{}}
				$\mu =2$
			\end{tabular}
			&
			\begin{tabular}
				{@{}l@{}}
				$p_S(n+2) < p < p_F(2)$
			\end{tabular}
			
			\\
			\bottomrule
		\end{tabular}
	\end{adjustbox}
\end{center}

\medskip

\begin{center}
	\footnotesize
	
	\begin{adjustbox}{width=\textwidth}
		
		\begin{tabular}{lllll}
			\toprule
			\multicolumn{5}{c}{\emph{Blow-up in finite time for $\beta=1$}}
			\\\midrule
			Authors & Dim. $n$ & Coefficient $\mu$ & Exponent $p$ & Lifespan $T_\e$
			\\\midrule[\heavyrulewidth]

		\begin{tabular}
			{@{}l@{}} Wakasugi \\ \cite{wak14,wak14-thesis}
		\end{tabular}
		& $n \ge 1$
		& \begin{tabular}
			{@{}l@{}} $\mu\ge1$ \\ $0<\mu<1$
		  \end{tabular}
	  	& \begin{tabular}
	  		{@{}l@{}} $1<p\le p_F(n)$ \\ $1<p<1+\frac{2}{n+\mu-1}$
	  	  \end{tabular}
  		& \begin{tabular}
  			{@{}l@{}} $\lesssim \e^{-{(p-1)}/{\gamma_F(p,n)}}$ \\ $\lesssim \e^{-{(p-1)}/{\gamma_F(p,n+\mu-1)}}$
  		  \end{tabular}
  		
  		\\\midrule
  		
  		\begin{tabular}
  			{@{}l@{}} D'Abbicco,\\Lucente,\\Reissig \cite{DLR15}
  		\end{tabular}
  		& \begin{tabular}
  			{@{}l@{}}
  			$n=1$ \\ $n = 2,3$
  		\end{tabular}
  		& $\mu=2$
  		& \begin{tabular}
  			{@{}l@{}}
  			$1<p \le p_F(1)$ \\ $1<p \le p_S(n+2)$
  		  \end{tabular}
  		&
     	
     	 \\\midrule
     	
     	 \begin{tabular}
     	 	{@{}l@{}}
     	 	Wakasa \cite{wakasa16} \\
     	 	Kato, \\ Takamura, \\ Wakasa \cite{KTW19}
     	\end{tabular}
     	 & $n=1$
     	 & $\mu=2$
     	 & \begin{tabular}
     	 	{@{}l@{}}
     	 	$1<p< p_F(1)$ \\ $p=p_F(1)$
     	 \end{tabular}
     	 & \begin{tabular}
     	 	{@{}l@{}}
     	 	$\sim \e^{-{(p-1)}/{\gamma_F(p,1)}}$ \\ $\sim \exp(C\e^{-(p-1)})$
     	 \end{tabular}

	      \\\midrule
	
	      \begin{tabular}
	      	{@{}l@{}}
	      	Imai, \\ Kato, \\ Takamura, \\ Wakasa \cite{IKTW19-arxiv}
	      \end{tabular}
	      & $n=2$
	      & $\mu=2$
	      & \begin{tabular}
	      	{@{}l@{}}
	      	$1<p< p_F(2)=p_S(2)$ \\ $p=p_F(2)=p_S(2)$
	      \end{tabular}
	      & \begin{tabular}
	      	{@{}l@{}}
	      	$\sim \e^{-{(p-1)}/{\gamma_F(p,2)}}$ \\ $\sim \exp(C\e^{-1/2})$
	      \end{tabular}

  		\\\midrule
  		
  		\begin{tabular}
  			{@{}l@{}}
  			Kato, \\ Sakuraba \cite{KS19}
  		\end{tabular}
  		& $n=3$
  		& $\mu=2$
  		& \begin{tabular}
  			{@{}l@{}}
  			$1<p< p_S(5)$ \\ $p=p_S(5)$
  		\end{tabular}
  		& \begin{tabular}
  			{@{}l@{}}
  			$\sim \e^{-{2p(p-1)}/{\gamma_S(p,5)}}$ \\ $\sim \exp(C\e^{-p(p-1)})$
  		\end{tabular}

       \\\midrule

       \begin{tabular}
       	{@{}l@{}}
       	Lai, \\ Takamura, \\ Wakasa \cite{LTW17}
       \end{tabular}
       & $n \ge 2$
       & $0<\mu<\frac{n^2+n+2}{2(n+2)}$
       & \begin{tabular}
       	{@{}l@{}}
       	$p_F(n) \le p < p_S(n+2\mu)$
       \end{tabular}
       & \begin{tabular}
       	{@{}l@{}}
       	$\lesssim \e^{-{2p(p-1)}/{\gamma_S(p,n+2\mu)}}$
       \end{tabular}

   		\\\midrule
   		
   		\begin{tabular}
   			{@{}l@{}}
   			Ikeda, \\ Sobajima \cite{IS18}
   		\end{tabular}
   		& \begin{tabular}
   			{@{}l@{}}
   			$n\ge1$
   			\end{tabular}
   		& \begin{tabular}
   			{@{}l@{}}
   			\vspace{2mm}
   			$0\le\mu< \frac{n^2+n+2}{n+2} $ \\
   			($\mu\neq0$ if $n=1$)
   		\end{tabular}
   		& \begin{tabular}
   			{@{}l@{}}
   			$p_F(n) < p \le p_S(n+\mu)$
   		\end{tabular}
   		& \begin{tabular}
   			{@{}l@{}}
   			$\lesssim \e^{-{2p(p-1)}/{\gamma_S(p,n+\mu)}-\delta}$
   			\\ \, if
   			$
   			\begin{cases}
   				n=1, \tfrac{2}{3} \le \mu < \tfrac{4}{3}
   				\\
   				n=1, 0<\mu<\tfrac{2}{3}, p \ge \tfrac{2}{\mu}
   				\\
   				n\ge2, p>p_S(n+2+\mu)
   			\end{cases}
   			$
			\vspace{1mm}
			\\
			$\lesssim \e^{-\frac{2(p-1)}{\mu} - \delta}$
			\\ \, if $n=1, 0<\mu<\tfrac{2}{3}, p<\tfrac{2}{\mu} $
			\vspace{1mm}
			\\
			$\lesssim \e^{-1- \delta}$
			\\ \, if $n\ge2, p<p_S(n+2+\mu) $
			\vspace{1mm}
			\\
			$\lesssim \exp(C\e^{-p(p-1)})$
			\\ \, if $p=p_S(n+\mu)$.
   		\end{tabular}
   		
   		\\\midrule
   		
   		\begin{tabular}
   			{@{}l@{}}
   			Tu, Lin \\ \cite{TL17,TL19}
   		\end{tabular}
   		& \begin{tabular}
   			{@{}l@{}}
   			$n\ge2$
   		\end{tabular}
   		& \begin{tabular}
   			{@{}l@{}}
   			$\mu>0$ \\ $0<\mu< \frac{n^2+n+2}{n+2} $
   		\end{tabular}
   		& \begin{tabular}
   			{@{}l@{}}
   			$1 < p<p_S(n+\mu)$ 
   			\\ $p=p_S(n+\mu)$
   		\end{tabular}
   		& \begin{tabular}
   			{@{}l@{}}
   			$\lesssim \e^{-{2p(p-1)}/{\gamma_S(p,n+\mu)}}$ \\ $\lesssim \exp(C\e^{-p(p-1)})$
   		\end{tabular}
		
		\\
		\bottomrule
	\end{tabular}

\end{adjustbox}
\end{center}

Observe that the special case $\mu=2$ was widely studied, starting from D'Abbicco, Lucente and Reissig \cite{DLR15}. The reason is that, if we exploit the Liouville transform
\begin{equation*}
	v(x,t) := (1+t)^{\mu/2} u(x,t)
\end{equation*}
in problem \eqref{eq:damped_b=1}, it turns out to be
\begin{equation*}
\left\{
\begin{aligned}
& v_{tt} - \Delta v + \frac{\mu(2-\mu)}{4(1+t)^{2}} v =  \frac{|v|^p}{(1+t)^{\mu(p-1)/2}}, \quad\text{in $\R^n\times (0,T)$},
\\
& v(x,0)=\e f(x), \quad v_t(x,0)=\e \left\{ \frac{\mu}{2} f(x) + g(x) \right\}, \quad x\in\R^n.
\end{aligned}
\right.
\end{equation*}
For $\mu=2$ the damping term disappears, making the analysis more manageable and related to the undamped wave equation.
From the works \cite{DL15,DLR15,IS18,Pal19-mmas,wak14} is now clear that the critical exponent for $\mu=2$ is $p_{crit}=\max\{p_F(n),p_S(n+2)\}$, with the lifespan estimates stated in low dimensions $n\le3$ by the works \cite{IKTW19-arxiv,KS19,KTW19,wakasa16}.

When $\mu\neq2$, it was observed that for small $\mu$ the problem is wave-like in the critical exponent and in the lifespan estimates, whereas it is heat-like for larger $\mu$. However, the exact threshold was still unclear. We conjecture, in accordance with Remarks 1.2 and 1.4 in \cite{IS18}, that the threshold value should be
\begin{equation*}
	\mu_* \equiv \mu_*(n) := \frac{n^2+n+2}{n+2},
\end{equation*}
and that the critical exponent is
\begin{equation}\label{eq:pmu}
	p_{crit} = p_{\mu}(n) :=
	\max\{p_F(n-[\mu-1]_-), p_S(n+\mu)\} =
	\begin{cases}
	p_S(n+\mu) & \text{if $0\le\mu<\mu_*$,}
	\\
	p_F(n) & \text{if $\mu \ge \mu_*$.}
	\end{cases}
\end{equation}
Here and in the following, $[x]_{\pm} = \tfrac{|x| \pm x}{2}$ indicates the positive and negative part functions respectively.

The blow-up part of this conjecture has already been proved, combining \cite{wak14} and \cite{IS18}. In our next theorem, which is a straightforward corollary of Theorem\til\ref{thm1}, we reconfirm the blow-up range and we give cleaner estimates for the lifespan in the subcritical case, obtaining improvements mainly in the $1$-dimensional case (see Remark\til\ref{rem:impr_massless}). We refer to Figure\til\ref{fig1} for a graphic representation of the results below.

\begin{theorem}\label{cor1}
	Let $\mu\ge0$ and $1< p < p_{\mu}(n)$, with $p_\mu(n)$ defined in \eqref{eq:pmu}.
	%
	%
	Assume that $f \in H^1(\R^n)$,
	$g \in L^2(\R^n)$ and
	\begin{equation*}
	f \ge 0, 
	\quad
	[\mu-1]_{+} f(x) + g(x) > 0. 
	\end{equation*}
	Suppose that $u$ is an energy solution of \eqref{eq:damped_b=1} on $[0,T)$ that satisfies
	\begin{equation*}
	\supp u \subset\{(x,t)\in\R^n\times[0,\infty) \colon |x|\le t+R\}
	\end{equation*}
	with some $R\ge1$.
	
	Then, there exists a constant $\e_1=\e_1(f,g,\mu,p, R)>0$
	such that the blow-up time $T_\e$ of problem \eqref{eq:damped_b=1}, for $0<\e\le\e_1$, has to satisfy:
	\begin{itemize}
		\item if $0 \le \mu<\mu_*$, then
			\begin{equation*}
			T_{\e} \lesssim
				\left\{
				\begin{aligned}
				& \e^{-(p-1)/\gamma_F(p,n-[\mu-1]_-)}
				&\quad&\text{if $1 < p \le \frac{2}{n-|\mu-1|}$,}
				\\
				& \e^{-2p(p-1)/\gamma_S(p,n+\mu)}
				&\quad&\text{if $\frac{2}{n-|\mu-1|} < p < p_\mu(n)$;}
				\end{aligned}
				\right.
			\end{equation*}
		
		\item if $\mu \ge \mu_*$, then
		\begin{equation*}
		T_{\e} \lesssim
		\e^{-(p-1)/\gamma_F(p,n)}
		=
		\e^{-\left[ 2/(p-1) -n \right]^{-1}}.
		\end{equation*}
		
	\end{itemize}
	
	Moreover, if $\mu=n=1$ and $1<p\le 2$ the estimate for $T_\e$ is improved by
	\begin{equation*}
	T_{\e} \lesssim
	\phi_0(\e)
	\end{equation*}
	where $\phi_0 \equiv \phi_0(\e)$ is the solution of
	\begin{equation*}
	\e \phi_0^{\frac{2}{p-1}-1}
	\ln(1+\phi_0) = 1.
	\end{equation*}
	
\end{theorem}

\begin{remark}
	Note that, if $n\ge3$ and $0\le \mu < n-1$, we can write the lifespan estimates in Theorem\til\ref{cor1} explicitly as
	\begin{equation*}
	T_{\e} \lesssim
	\left\{
	\begin{aligned}
	& \e^{-2p(p-1)/\gamma_S(p,n+\mu)}
	&&\text{
		\begin{tabular}
		{@{}l@{}}
		if $0\le\mu\le n-1$ or
		\\
		if $n-1<\mu<\mu_*$ and $\frac{2}{n-\mu+1} < p < p_\mu(n)$,
		\end{tabular}
	}
	\\
	& \e^{-(p-1)/\gamma_F(p,n)}
	&&\text{
		\begin{tabular}
		{@{}l@{}}
		if $n-1<\mu<\mu_*$ and $1 < p \le \frac{2}{n-\mu+1}$.
		\end{tabular}
	}
	\end{aligned}
	\right.
	\end{equation*}
\end{remark}

\begin{remark}\label{rem:impr_massless}
	Comparing the lifespan estimates in Theorem\til\ref{cor1} with the known results summarized in the above table \lq\lq Blow-up in finite time for $\beta=1$'', we remark that the heat-like estimates for $n\ge1$ were already proved by Wakasugi \cite{wak14-thesis}, whereas the wave-like ones for $n\ge2$ by Tu and Lin \cite{TL17}. The wave-like estimates for $n=1$ were almost obtained by Ikeda and Sobajima \cite{IS18} for $p_F(n) \le p < p_S(n+\mu)$, with a loss in the exponent given by a constant $\delta>0$.
	
	Hence our improvements are given by the wave-like estimates if $n=1$ and by the logarithmic gain $T_\e \lesssim \phi_0(\e)$ if $n=\mu=1$ and $1<p\le 2$.
	Moreover, about the wave-like estimates for $n\ge2$, in \cite{TL17} the initial data are supposed to be non-negative, whereas our conditions on the initial data are less restrictive.
	
	Anyway, our approach is different and based on an iteration argument rather than on a test function method.
\end{remark}

\begin{remark}\label{rem:subcrit-conj-mu=0}
	We conjecture that the lifespan estimates in Theorem\til\ref{cor1} are indeed optimal, except on the \lq\lq\emph{transition curve}'' (in the $(p,\mu)$-plane) from the wave-like to the heat-like zone, given by
	\begin{equation*}
		p=\frac{2}{n-|\mu-1|}
		\quad
		\text{for $0\le\mu\le\mu_*$ and $1<p \le p_\mu(n)$.}
	\end{equation*}
	On this line, the identity
	\begin{equation*}
		2p \, \gamma_F(p,n-[\mu-1]_-) = \gamma_S(p,n+\mu)
	\end{equation*}
	holds true and here we expect a logarithmic gain, as already obtained for the case $p=2$, $\mu= n=1$ in the previous theorem, and for the case $n=p=2$, $\mu=0$ for the wave equation (see Subsection\til\ref{sub:HvsW}). As we see from \cite{IKTW19-arxiv,KS19,KTW19,wakasa16} the conjecture holds true if $\mu=2$ and $n\le3$.
\end{remark}

\begin{remark}\label{rem:crit-conj-mu=0}
	In this work we do not treat the critical case, but, to conclude our prospectus, it is natural to conjecture that
	\begin{equation*}
		T_\e \sim
		\left\{
		\begin{aligned}
		& \exp\left( C \e^{-p(p-1)}\right)
		&&\text{if $0\le \mu < \mu_*$ and $p=p_\mu(n)=p_S(n+\mu)$,}
		\\
		& \exp\left( C \e^{-(p-1)}\right)
		&&\text{if $\mu > \mu_*$ and $p=p_\mu(n)=p_F(n)$,}
		\end{aligned}
		\right.
	\end{equation*}
	for some constant $C>0$. We refer to \cite{IS18,TL19} for the wave-like lifespan estimate from above in the critical case and to \cite{IKTW19-arxiv,KS19,KTW19,wakasa16} for the proof of the conjecture if $\mu=2$ and $n=1,3$.
	
	However, we expect a different behaviour if $\mu=\mu_*$ and $p=p_{\mu_*}(n)$, that is when the transition curve from Remark\til\ref{rem:subcrit-conj-mu=0} intersects the blow-up curve. This expectation is  motivated from \cite{IKTW19-arxiv}, where the authors prove for $n=\mu=\mu_*=p_F(2)=p_S(4)=2$ that $T_\e \sim \exp(C \e^{-1/2})$, which is neither a wave-like critical lifespan, nor a heat-like one.
\end{remark}

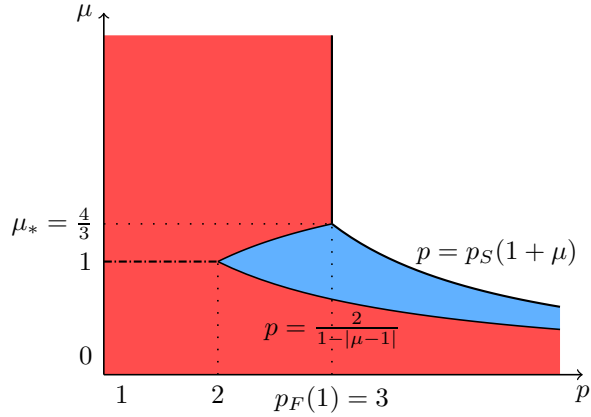
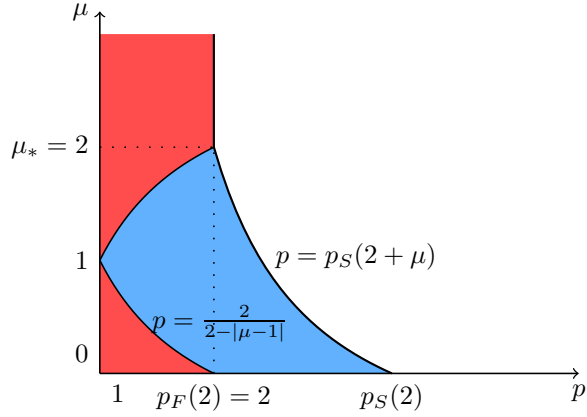
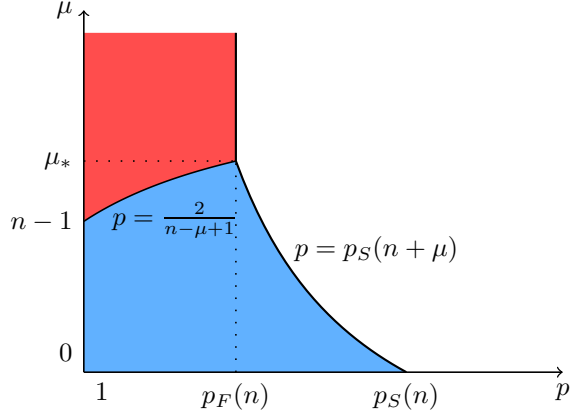
\begin{figure}
	\centering
	\subfloat[][Case $n=1$.]{
		\begin{tikzpicture}[scale=1.5, semithick]
		
		\begin{scope}[font=\footnotesize]
		
		\fill[wavecol]
		plot [domain=5:3] ( \x-1, {2*(\x+1)/(\x-1)/\x} )
		-- plot [domain=3:2] ( \x-1, {2-2/\x} )
		-- plot [domain=2:5] ( \x-1, {2/\x} )
		-- cycle;
		
		\fill[heatcol]
		(4,0) -- (0,0) -- (0,3) -- (2,3) -- (2,4/3)
		-- plot [domain=3:2] ( \x-1, {2-2/\x} )
		-- plot [domain=2:5] ( \x-1, {2/\x} )
		-- cycle;
		
		
		
		\draw[loosely dotted] (0,4/3) node [left] {$\mu_*=\frac{4}{3}$} -- (2,4/3);
		
		\draw[loosely dotted] (2,0) node [below] {$p_F(1)=3$} -- (2,4/3);
		
		\draw[loosely dotted] (1,0) node [below] {$2$} -- (1,1);
		
		\draw[thick] (2,4/3) -- (2,3);
		
		\draw[densely dash dot, thick] (0,1) node [left] {$1$} -- (1,1);
		
		\draw[thick]
		plot [domain=3:5] ( \x-1, {2*(\x+1)/(\x-1)/\x} )
		;	
		
		\draw
		plot[domain=2:3] ( \x-1, {2-2/\x} );	
		
		\draw
		plot [domain=2:5] ( \x-1, {2/\x} )
		;

		
		
		\node [above left] at (0,0) {$0$};
		\node [below right] at (0,0) {$1$};
		
		\draw[->]  (0,0) -- (0,3.2) node [left] {$\mu$};
		\draw[->]  (0,0) -- (4.2,0) node [below] {$p$};

		\node [below] at (3.45,1.3) {$p=p_S(1+\mu)$};
		\node [below left] at (2.7,0.65) {$p=\tfrac{2}{1-|\mu-1|}$};
		
		
		\node[left,opacity=0] at (0,0) {$\mu_*=\frac{4}{3}$};
		\node[left,opacity=0] at (0,0) {$\mu_*=2$};
		\node[left,opacity=0] at (0,0) {$n-1$};
		
		\end{scope}
		\end{tikzpicture}
	}
	\quad
	\subfloat[][Case $n=2$.]{
		\begin{tikzpicture}[scale=1.5, semithick]
		
		\begin{scope}[font=\footnotesize]
		
		\fill[wavecol]
		plot [domain=(3+sqrt(17))/2:2] ( \x-1, {(-\x*\x+3*\x+2)/(\x*\x-\x)} )
		-- plot [domain=2:1] ( \x-1, {3-2/\x} )
		-- plot [domain=1:2] ( \x-1, {-1+2/\x} )
		-- cycle;

		\fill[heatcol]
		(0,3) -- (1,3) -- (1,2)
		-- plot [domain=2:1] ( \x-1, {3-2/\x} )
		-- cycle;

		\fill[heatcol]
		(0,0) -- (0,1)
		-- plot [domain=1:2] ( \x-1, {-1+2/\x} )
		-- cycle;
		
		
		
		\draw[loosely dotted] (0,2) node [left] {$\mu_*=2$} -- (1,2);
		
		\draw[loosely dotted] (1,0) node [below] {$p_F(2)=2$} -- (1,2);
		
		\draw[thick] (1,2) -- (1,3);
		
		\draw[thick]
		plot [domain=2:(3+sqrt(17))/2] ( \x-1, {(-\x*\x+3*\x+2)/(\x*\x-\x)} );	
		
		\draw
		plot [domain=1:2] ( \x-1, {3-2/\x} );	
		
		\draw
		plot [domain=1:2] ( \x-1, {-1+2/\x} );

		\node [left] at (0,1) {$1$};
		
		\node [below] at (2.56155 , 0)
		{$p_S(2)$};

		
		\node at (0,0) {};
		\node [above left] at (0,0) {$0$};
		\node [below right] at (0,0) {$1$};
		
		\draw[->]  (0,0) -- (0,3.2) node [left] {$\mu$};
		\draw[->]  (0,0) -- (4.2,0) node [below] {$p$};

		\node [below] at (2.25,1.25) {$p=p_S(2+\mu)$};
		\node [below left] at (1.75,0.7) {$p=\tfrac{2}{2-|\mu-1|}$};
		
		\node[left,opacity=0] at (0,0) {$\mu_*=\frac{4}{3}$};
		\node[left,opacity=0] at (0,0) {$\mu_*=2$};
		\node[left,opacity=0] at (0,0) {$n-1$};
		
		\end{scope}
		\end{tikzpicture}
	}
	\quad
	\subfloat[][Case $n\ge3$.]{
		\begin{tikzpicture}[scale=1.5, semithick]
		
		\begin{scope}[font=\footnotesize]
		
		\fill[wavecol]
		(0,0)
		-- plot [domain=1:1+2/3] ( \x*2-2, {(4-2/\x)*2/3} )
		-- plot [domain=1+2/3:1+sqrt(32)/4]
		( \x*2-2, {((-2*\x*\x+4*\x+2)/(\x*\x-\x))*2/3} )		
		-- cycle;

		\fill[heatcol]
		(2/3*2,3) -- (0,3) -- (0,2/2)
		-- plot [domain=1:1+2/3] ( \x*2-2, {(4-2/\x)*2/3} )
		-- cycle;
		

		
		\draw[loosely dotted] (0,1*2/3+9/5*2/3) node [left] {$\mu_*$} -- (2/3*2,1*2/3+9/5*2/3);
		
		\draw[loosely dotted] (2/3*2,0) node [below] {$p_F(n)$} -- (2/3*2,1*2/3+9/5*2/3);
		
		\draw[thick] (2/3*2,1*2/3+9/5*2/3) -- (2/3*2,3);
		
		\draw[thick]
		plot [domain=1+2/3:1+sqrt(32)/4]
		( \x*2-2, {((-2*\x*\x+4*\x+2)/(\x*\x-\x))*2/3} );	
		
		\draw
		plot [domain=1:1+2/3] ( \x*2-2, {(4-2/\x)*2/3} );

		\node [left] at (0,2*2/3) {$n-1$};
		
		\node [below] at (1.41421*2, 0) {$p_S(n)$};

		
		\node at (0,0) {};
		\node [above left] at (0,0) {$0$};
		\node [below right] at (0,0) {$1$};
		
		\draw[->]  (0,0) -- (0,3.2) node [left] {$\mu$};
		\draw[->]  (0,0) -- (4.2,0) node [below] {$p$};

		\node [above right] at (1.75,0.65*2*2/3) {$p=p_S(n+\mu)$};
		\node [below] at (0.8,1.2*2*2/3) {$p=\tfrac{2}{n-\mu+1}$};
		
		\node[left,opacity=0] at (0,0) {$\mu_*=\frac{4}{3}$};
		\node[left,opacity=0] at (0,0) {$\mu_*=2$};
		\node[left,opacity=0] at (0,0) {$n-1$};
		
		\end{scope}
		\end{tikzpicture}
	}
	\caption{In this figure we collect the results from Theorem \ref{cor1}. If $(p,\mu)$ is in the blue area, we have that $T_\e \lesssim \e^{-2p(p-1)/\gamma_S(p,n+\mu)}$ and hence the lifespan estimate is wave-like.
		Otherwise, if $(p,\mu)$ is in the red area, then $T_\e \lesssim \e^{-(p-1)/\gamma_F(p,n-[\mu-1]_-)}$ and the lifespan estimate is heat-like.
		In the case $n=1$, the dash-dotted line given by $\mu=1$, $1<p\le 2$ highlights the improvement $T_\e \lesssim \phi_0(\e)$.}
	\label{fig1}
\end{figure}


\subsection{Wave equation with scale-invariant damping and mass}

Finally, we return to our main problem \eqref{eq:main_problem}. The scale-invariant damped and massive wave equation was studied by A. Palmieri as object of his doctoral dissertation \cite{Pal18-thesis}, under the supervision of M. Reissig. However, as far as we know, the research of the lifespan estimates in case of blow-up is still underdeveloped.

A key value for the study of this problem is
\begin{equation*}
\delta \equiv \delta(\mu_1,\mu_2) := (\mu_1-1)^2-4\mu_2,
\end{equation*}
which, roughly speaking, quantify the interaction between the damping and the mass term. Indeed, if $\delta\ge0$, the damping term is predominant and we observe again a competition between the wave-like and heat-like behaviours. In particular, the critical exponent seems to be wave-like for small positive values of $\delta$, while it is heat-like for large ones.
If on the contrary $\delta<0$, the mass term has more influence and the equation becomes of Klein-Gordon type. To see this, apply again the Liouville transform $v(x,t) := (1+t)^{\mu_1/2} u(x,t)$ to problem \eqref{eq:main_problem}, which therefore becomes
\begin{equation}\label{liouville}
	\left\{
	\begin{aligned}
	& v_{tt} - \Delta v + \frac{(1-\delta)/4}{(1+t)^{2}} v =  \frac{|v|^p}{(1+t)^{\mu_1(p-1)/2}}, \quad\text{in $\R^n\times (0,T)$}, \\
	& v(x,0)=\e f(x), \quad v_t(x,0)=\e \left\{ \frac{\mu_1}{2} f(x) + g(x) \right\}, \quad x\in\R^n.
	\end{aligned}
	\right.
\end{equation}
In the following, we will consider only the case $\delta\ge0$.

Let us start by collecting some known results.
From \cite{NPR17,Pal18,PR18}, we know that for $\mu_1,\mu_2>0$ and $\delta\ge(n+1)^2$ the critical exponent for problem \eqref{eq:main_problem} is the shifted Fujita exponent
\begin{equation*}
	p_{crit} = p_F\left( n+\frac{\mu_1-1-\sqrt{\delta}}{2} \right).
\end{equation*}
On the contrary, from \cite{Pal19-mmas,Pal19-newtools}, in the special case $\delta=1$ and under radial symmetric assumptions for $n\ge3$, Palmieri proved that the critical exponent is
\begin{equation*}
	p_{crit} = p_S\left( n+\mu_1 \right).
\end{equation*}
The case $\delta=1$ is clearly the analogous of the case $\mu=2$ for the scale-invariant damped wave equation without mass: under this assumption
we see from \eqref{liouville} that the equation can be transformed into a wave equation without damping and mass and with a suitable nonlinearity.
In \cite{PR19-jde}, Palmieri and Reissig proved, by using the Kato's lemma and Yagdjian integral transform, a blow-up result for $\delta \in (0,1] $, and showed a competition between the shifted Fujita and Strauss exponents. Indeed, they obtained the blow-up result for
\begin{equation*}
	1 < p \le \max\left\{ p_F\left(n+\frac{\mu_1-1-\sqrt{\delta}}{2}\right) , p_S(n+\mu_1) \right\}
\end{equation*}
except for the critical case $p=p_S(n+\mu_1)$ in dimension $n=1$.
Finally, Palmieri and Tu in \cite{PT19}, under suitable sign assumption on the initial data and for $\mu_1,\mu_2, \delta$ non-negative, established a blow-up result for $1<p \le p_S(n+\mu_1)$ and furthermore the following lifespan estimates:
\begin{equation*}
	T_\e \lesssim
	\left\{
	\begin{aligned}
	& \e^{-2p(p-1)/\gamma_S(p,n+\mu_1)}
	&& \text{if $1<p<p_S(n+\mu_1)$,}
	\\
	& \exp(C \e^{-p(p-1)})
	&& \text{if $p=p_S(n+\mu_1)$ and $p>\frac{2}{n-\sqrt{\delta}}$.}
	\end{aligned}	
	\right.
\end{equation*}
They used an iteration argument based on the technique of double multiplier for the subcritical case and a version of test function method developed by Ikeda and Sobajima \cite{IS18} for the critical case.
Of course, we refer to the works by Palmieri and to his doctoral thesis for a more detailed background.

We present now our main result, concerning the blow-up of \eqref{eq:main_problem} for $\mu_1,\mu_2\in\R$ and $\delta\ge0$ and the upper bound for the lifespan estimates.

Firstly, let us introduce the value
	\begin{equation}
	\label{def:d*1}
	d_*(\nu) :=	
	\left\{
	\begin{aligned}
	&\frac{1}{2} \left(-1- \nu + \sqrt{\nu^2+10\nu-7} \right)
	&\quad\text{if $\nu > 1$,}
	\\
	&0
	&\quad\text{if $\nu \le 1$,}
	\end{aligned}
	\right.
	\end{equation}
	and set for the simplicity
	\begin{equation}
	\label{def:d*2}
		d_* := d_*(n+\mu_1) \in [0,2).
	\end{equation}
	Observe that, if $n+\mu_1>1$, then
	\begin{equation}
	\label{d*rel}
	\begin{split}
	\sqrt{\delta} = n - d_*
	& \Longleftrightarrow
	\gamma_S(p,n+\mu_1) = 2\, \gamma_F \left(p, n + \frac{\mu_1 -1- \sqrt{\delta}}{2}\right) =0 \\
	& \Longleftrightarrow
	p_S(n+\mu_1) = p_F \left( n +\frac{\mu_1 -1 - \sqrt{\delta}}{2}\right) = \frac{2}{n - \sqrt{\delta}}.
	\end{split}
	\end{equation}
%
%
The following result holds.

\begin{theorem}\label{thm1}
	Let $\mu_1,\mu_2 \in\R$, $\delta\ge0$ and $1< p < p_{\mu_1,\delta}(n)$, with
	\begin{equation}
	\label{pch>0}
	p_{\mu_1,\delta}(n) :=
	\max\left\{
	p_F\left(n+\frac{\mu_1-1-\sqrt{\delta}}{2}\right), \,
	p_S\left(n+\mu_1\right) \right\}.
	\end{equation}
	Assume that $f \in H^1(\R^n)$,
	$g \in L^2(\R^n)$ and
	\begin{equation}
	\label{hpindata}
		f \ge 0, 
		\quad
		h > 0, 
		\quad\text{where}\quad
		h:=\frac{\mu_1-1+\sqrt{\delta}}{2}f+g.
	\end{equation}
	Suppose that $u$ is an energy solution of \eqref{eq:main_problem} on $[0,T)$ that satisfies
	\begin{equation}
	\label{support}
	\supp u \subset\{(x,t)\in\R^n\times[0,\infty) \colon |x|\le t+R\}
	\end{equation}
	with some $R\ge1$.
	
	Then, there exists a constant $\e_2=\e_2(f,g,\mu_1,\mu_2,n,p, R)>0$
	such that the blow-up time $T_\e$ of problem \eqref{eq:main_problem}, for $0<\e\le\e_2$, has to satisfy:
	\begin{itemize}
		\item if $\sqrt{\delta} \le n-2$, then
		\begin{equation*}
		T_{\e} \lesssim \e^{-2p(p-1)/\gamma_S(p,n+\mu_1)};
		\end{equation*}
		
		\item if $n-2 < \sqrt{\delta} < n-d_*(n+\mu_1)$, then
		\begin{equation*}
			T_\e \lesssim
			\left\{
			\begin{aligned}
			&\phi(\e)
			&&\text{if $1<p\le \frac{2}{n-\sqrt{\delta}}$,}
			\\
			&\e^{-2p(p-1)/\gamma_S(p,n+\mu_1)}
			&&\text{if $\frac{2}{n-\sqrt{\delta}} < p < p_{\mu_1,\delta}(n)$,}
			\end{aligned}
			\right.
		\end{equation*}
		where $\phi \equiv \phi(\e)$ is the solution of
		\begin{equation*}\label{def:phi}
		\e \phi^{\frac{\gamma_F\left(p,n+(\mu_1-1-\sqrt{\delta})/2\right)}{p-1}}
		\ln(1+\phi)^{1-\sgn\delta} = 1;
		\end{equation*}
		
		\item if $\sqrt{\delta} \ge n - d_*(n+\mu_1)$, then
		\begin{equation*}
		T_{\e} \lesssim \phi(\e).
		\end{equation*}
	\end{itemize}
	If in particular $\delta>0$, then		
	\begin{equation*}
	\phi(\e) = \e^{-(p-1)/\gamma_F(p,n+(\mu_1-1-\sqrt{\delta})/2)}
	=
	\e^{-\left[ 2/(p-1) - n - (\mu_1-1-\sqrt{\delta})/2 \right]^{-1}}.
	\end{equation*}
\end{theorem}

Here and in the following, the sign function is defined as $\sgn x = \tfrac{|x|}{x}$ if $x\neq0$, whereas $\sgn x =0$ if $x=0$.

\begin{remark}
	We can write the exponent in \eqref{pch>0} explicitly as
	\begin{equation*}
	\begin{aligned}
	p_{\mu_1,\delta}(n) =
	\left\{
	\begin{aligned}
	& p_S\left(n+\mu_1\right)
	&\quad&\text{
		\begin{tabular}
		{@{}l@{}}
		if $n+\mu_1>1$, $\sqrt{\delta} \le n-d_*$,
		\end{tabular}
	}
	\\
	& p_F\left(n+\frac{\mu_1-1-\sqrt{\delta}}{2}\right)
	&\quad&\text{
		\begin{tabular}
		{@{}l@{}}
		if $n+\mu_1>1$, $n-d_* < \sqrt{\delta} < 2n+\mu_1-1$,
		\end{tabular}
	}
	\\
	& +\infty
	&\quad&\text{
		\begin{tabular}
		{@{}l@{}}
		if $n+\mu_1>1$, $\sqrt{\delta} \ge 2n+\mu_1-1$
		\\
		or if $n+\mu_1 \le 1$.
		\end{tabular}
	}
	\end{aligned}
	\right.
	\end{aligned}
	\end{equation*}
\end{remark}

\begin{remark}
	Note that, setting the mass coefficient $\mu_2=0$ and the damping coefficient $\mu_1=\mu>0$, then $\sqrt{\delta}=|\mu-1|$ and
	\begin{equation*}
	\sqrt{\delta} \le n-d_*(n+\mu)
	\Longleftrightarrow
	0< \mu \le \mu_*.
	\end{equation*}
	It is straightforward to check that, by imposing $\mu_2=0$, the results in Theorem\til\ref{thm1} coincide with those in Theorem\til\ref{cor1}.
\end{remark}

\begin{remark}\label{rem:conj1}
	Analogously as in Remark\til\ref{rem:subcrit-conj-mu=0}, we conjecture that $p_{\mu_1,\delta}(n)$ defined in \eqref{pch>0} is indeed the critical exponent and that the lifespan estimates presented in Theorem\til\ref{thm1} are optimal, except on the \lq\lq\emph{transition surface}''(in the $(p,\mu_1,\delta)$-space) defined by
	\begin{equation}\label{trans_p}
		p= \frac{2}{n-\sqrt{\delta}}
		\quad\text{for $n-2<\sqrt{\delta}<n-d_*(n+\mu_1)$ and $1<p \le p_{\mu_1,\delta}(n)$,}
	\end{equation}
	on which we expect a logarithmic gain.
	
	The exponent $p=\frac{2}{n-\sqrt{\delta}}$ already emerged in Palmieri and Tu \cite{PT19}, but as a technical condition. We underline that this exponent comes out to be the solution of the equation
	\begin{equation*}
	2p \, \gamma_F\left(p, n+\frac{\mu_1-1-\sqrt{\delta}}{2}\right) = \gamma_S(p,n+\mu_1)
	\end{equation*}
	when $n-2 < \sqrt{\delta} < n-d_*(n+\mu_1)$.
\end{remark}

\begin{remark}\label{rem:conj2}
	Similarly as in Remark\til\ref{rem:crit-conj-mu=0}, we expect that, if $p=p_{\mu_1,\delta}(n)$, then
	\begin{equation*}
		T_\e \sim
		\left\{
			\begin{aligned}
				& \exp\left( C \e^{-p(p-1)}\right)
				&&
				\text{
					\begin{tabular}
					{@{}l@{}}
					if $n+\mu_1>1$
					and $\sqrt{\delta}< n-d_*$,
					\end{tabular}
				}
				\\
				& \exp\left( C \e^{-(p-1)}\right)
				&&
				\text{
					\begin{tabular}
					{@{}l@{}}
					if $n+\mu_1>1$ 
					and $n-d_* < \sqrt{\delta}<2n+\mu_1-1$,
					\end{tabular}
				}
			\end{aligned}
		\right.
	\end{equation*}
	for some constant $C>0$. See \cite{PT19} for the proof of the wave-like upper bound of the lifespan estimate in the critical case. Moroever, if $\sqrt{\delta}=n-d_*(n+\mu_1)$ and $p=p_{\mu_1,\delta}(n)$, we expect a different lifespan estimate, as in the massless case.
\end{remark}

\subsection{Different lifespans for different initial conditions}
\label{subsec:h=0}

In Theorems\til\ref{cor1} \& \ref{thm1} we impose the condition on the initial data
\begin{equation*}\label{h}
	h = \frac{\mu_1-1+\sqrt{\delta}}{2} f + g >0.
\end{equation*}
One could ask if this is only a technical condition, but it turns out that this is not the case: if we impose $h=0$, the lifespan estimates change drastically. This phenomenon was recently taken in consideration also in the works by Imai, Kato, Takamura and Wakasa \cite{IKTW19-aa,IKTW19-arxiv,KTW19}.

Let us return to the wave equation
\begin{equation*}
\left\{
\begin{aligned}
& u_{tt} - \Delta u = |u|^p, \quad\text{in $\R^n\times (0,T)$}, \\
& u(x,0)=\e f(x), \quad u_t(x,0)=\e g(x), \quad x\in\R^n.
\end{aligned}
\right.
\end{equation*}
Since $\mu_1=\mu_2=0$, in this case the condition $h=0$ is equivalent to $g=0$. Indeed, under the assumption
\begin{equation*}
	\int_{\R^n} g(x) dx =0,
\end{equation*}
collecting the results from the works \cite{IKTW19-aa,	LZ14,Lin90,LS96,T15,TW11,Zhou92-cam,Zhou92-jde,Zhou93}, we have that, for $n\ge1$, the following lifespan estimates holds:
\begin{equation*}
	T_\e \sim
	\left\{
	\begin{aligned}	
	& \e^{-2p(p-1)/\gamma_S(p,n)}
	&\quad&\text{if $1<p<p_S(n)$,}
	\\
	& \exp\left(C \e^{-p(p-1)}\right)
	&\quad&\text{if $p=p_S(n)$,}
	\end{aligned}
	\right.
\end{equation*}
excluding the critical case $p=p_S(n)$ for $n\ge9$ and without radial symmetry assumptions. We refer to the Introduction by Imai, Kato, Takamura and Wakasa \cite{IKTW19-aa} for a detailed background on these results. What is interesting is the fact that now we observe always a wave-like lifespan. This is in contrast with the estimates presented in Subsection\til\ref{sub:HvsW}, where, under the assumption
\begin{equation*}
	\ints g(x)dx>0,
\end{equation*}
we have heat-like lifespans in low dimensions, more precisely if $n=1$ or if $n=2$ and $1<p\le 2$, with a logarithmic gain if $n=p=2$.

Let us consider now the Cauchy problem for the scale-invariant damped wave equation \eqref{eq:damped} with $\mu=2$, that is
\begin{equation*}
\left\{
\begin{aligned}
& u_{tt} - \Delta u + \frac{2}{1+t} u_t = |u|^p, \quad\text{in $\R^n\times (0,T)$}, \\
& u(x,0)=\e f(x), \quad u_t(x,0)=\e g(x), \quad x\in\R^n.
\end{aligned}
\right.
\end{equation*}
Since $\mu_1=2$ and $\mu_2=0$, the condition $h=0$ is equivalent to $f+g=0$. In low dimensions $n=1$ and $n=2$, Kato, Takamura and Wakasa \cite{KTW19} and Imai, Kato, Takamura and Wakasa \cite{IKTW19-arxiv} proved that, if the initial data satisfy
\begin{equation*}
	\ints \{ f(x) + g(x) \} dx =0,
\end{equation*}
then the lifespan estimates in $1$-dimensional case are
\begin{equation*}
	T_\e \sim
	\left\{
	\begin{aligned}
	& \e^{-2p(p-1)/\gamma_S(p,3)}
	&& \text{if $1<p<2$,}
	\\
	& b(\e)
	&& \text{if $p=2$,}
	\\
	& \e^{-p(p-1)/\gamma_F(p,1)}
	&& \text{if $2<p<p_F(1)$,}
	\\
	& \exp(C \e^{-p(p-1)})
	&& \text{if $p=p_F(1)=3$,}
	\end{aligned}
	\right.
\end{equation*}
where $b \equiv b(\e)$ satisfies the equation $\e^2 b \log(1+b)=1$, and in $2$-dimensional case are
\begin{equation*}
	T_\e \sim
	\left\{
	\begin{aligned}
	& \e^{-2p(p-1)/\gamma_S(p,4)}
	&& \text{if $1<p<p_F(1)=p_S(4)=2$,}
	\\
	& \exp(C \e^{-2/3})
	&& \text{if $p=p_F(2)=p_S(4)=2$.}
	\end{aligned}
	\right.
\end{equation*}
These estimates are greatly different from the ones presented in Subsection\til\ref{sub:damped}, which hold under the assumption 
\begin{equation*}
	\ints \{f(x)+g(x)\} \neq 0.
\end{equation*}
In dimension $n=1$, we have no more a heat-like behaviour, but a wave-like one appears for $p<2$, whereas for $p>2$ we have a mixed-like behaviour, accordingly with the notation introduced in Subsection\til\ref{sub:HvsW}. Indeed, in the latter case, even if the lifespan is related to the heat-like one, an additional $p$ appears. In dimension $n=2$, we have no more a heat-like behaviour, but a wave-like one. The strange exponent in the critical lifespan can be explained by the same phenomenon underlined in Remark\til\ref{rem:crit-conj-mu=0}.

We are ready to exhibit our results, which give upper lifespan estimate in the subcritical case when $h=0$. It is easy to see that our estimates coincide with the ones just showed above in the respective cases. Going on with the exposition followed until now, we will present firstly the particular massless case, then the more general one where also the mass is considered. For the simplicity, we will consider only non-negative damping coefficients.

Let us introduce the exponent
\begin{equation}
\label{def:p*}
p_* \equiv p_*(n+\mu_1,n-\sqrt{\delta}) :=
\left\{
\begin{aligned}
&1+\frac{n-\sqrt{\delta}+2}{n+\mu_1-1},
&\quad\text{if $n+\mu_1 \neq 1$,}
\\
& +\infty,
&\quad\text{if $n+\mu_1 = 1$,}
\end{aligned}
\right.
\end{equation}
and observe that, for $p>1$ and $n+\mu_1 \neq 1$,
\begin{equation}
\label{p*rel}
p=p_*
\Longleftrightarrow
\gamma_S(p,n+\mu_1) =
2\, \gamma_F\left(p,n+\frac{\mu_1-1-\sqrt{\delta}}{2}\right).
\end{equation}

The following results hold. See Figure\til\ref{fig2} for a graphic representation of the claim in Theorem\til\ref{cor2}.

\begin{theorem}\label{cor2}
	Let $\mu \ge 0$ and $1< p < p_{\mu}(n)$, with $p_{\mu}(n)$ as in Theorem \ref{cor1}.
	Assume that $f \in H^1(\R^n)$,
	$g \in L^2(\R^n)$ and
	\begin{equation*}
	f > 0, 
	\quad
	[\mu-1]_{+} f(x) + g(x) = 0. 
	\end{equation*}
	Suppose that $u$ is an energy solution of \eqref{eq:damped_b=1} on $[0,T)$ that satisfies
	\eqref{support}
	for some $R\ge1$.
	
	Then there exists a constant $\e_3=\e_3(f,g,\mu,p, R)>0$
	such that the blow-up time $T_\e$ of problem \eqref{eq:damped_b=1}, for $0<\e\le\e_3$, has to satisfy:
	\begin{itemize}
		\item if $0 \le \mu \le \mu_*$, then
		\begin{equation*}
		T_{\e} \lesssim
		\e^{-2p(p-1)/\gamma_S(p,n+\mu)};
		\end{equation*}
		
		\item if $\mu_* < \mu < n+3$, then
		\begin{equation*}
		T_{\e} \lesssim
		\left\{
		\begin{aligned}
		& \e^{-2p(p-1)/\gamma_S(p,n+\mu)},
		&\quad&\text{if $1<p<p_*$},
		\\
		& \sigma_0(\e),
		&\quad&\text{if $p=p_*$},
		\\
		& \e^{-p(p-1)/\gamma_F(p,n)},
		&\quad&\text{if $p_*<p<p_{\mu}(n)$},
		\end{aligned}
		\right.
		\end{equation*}
		where $\sigma_0 \equiv \sigma_0(\e)$ is the solution of
		\begin{equation*}
		\e^p \sigma_0^{\frac{2}{p-1}-n} \ln(1+\sigma_0) = 1
		\end{equation*}
		and
		\begin{equation*}
		p_* = 1 + \frac{n-\mu+3}{n+\mu-1};
		\end{equation*}
		
		\item if $\mu \ge n+3$, then
		\begin{equation*}
		T_{\e} \lesssim
		\e^{-p(p-1)/\gamma_F(p,n)}.
		\end{equation*}
	\end{itemize}
	
	Moreover, if $n=1$, $0<\mu<2$ and
	\begin{equation*}
	1< p < \frac{2}{1+|\mu-1|},
	\end{equation*}
	then the estimate for the blow-up time $T_\e$ is improved by
	\begin{equation*}
	T_{\e} \lesssim \e^{-(p-1)/\gamma_F\left(p, 1 + [\mu-1]_{+} \right)}.
	\end{equation*}
\end{theorem}


\begin{figure}	
	\centering
	\subfloat[][Case $n=1$.]{
		\begin{tikzpicture}[scale=1.5, semithick]
		
		\begin{scope}[font=\footnotesize]
		
		\fill[wavecol]
		(4,0) -- (0,0)
		-- plot [domain=1:2] ( \x-1, {2-2/\x} )
		-- plot [domain=2:1] ( \x-1, {2/\x} )
		-- (0,2) -- (0,4)
		-- plot[domain=1:3] (\x-1, {4/\x})
		-- plot [domain=3:5] ( \x-1, {2*(\x+1)/(\x-1)/\x} )
		-- cycle;
		
		\fill[heatcol]
		plot [domain=1:2] ( \x-1, {2-2/\x} )
		-- plot [domain=2:1] ( \x-1, {2/\x} )
		--cycle;
		
		\fill[heatcol]
		(0,4) -- (0,4.1) -- (2,4.1) -- (2,4/3)
		-- plot[domain=3:1] (\x-1, {4/\x})
		-- cycle;
		
		\fill[mixedcol]
		(0,4) -- (0,4.1) -- (2,4.1) -- (2,4/3)
		-- plot[domain=3:1] (\x-1, {4/\x})
		-- cycle;
		
		
		\draw[loosely dotted] (0,4/3) node [left] {$\mu_*=\frac{4}{3}$} -- (2,4/3);
		
		\draw[loosely dotted] (2,0) node [below] {$p_F(1)=3$} -- (2,4/3);
		
		\draw[loosely dotted] (1,0) node [below] {$2$} -- (1,1);
		
		\draw[thick] (2,4/3) -- (2,4.1);
		
		\draw[loosely dotted] (0,1) node [left] {$1$} -- (1,1);
		
		\node[left] at (0,4) {$4$};
		
		\node[left] at (0,2) {$2$};
		
		\draw[thick]
		plot [domain=3:5] ( \x-1, {2*(\x+1)/(\x-1)/\x} ) ;	
		
		\draw[densely dash dot, thick]
		plot[domain=1:3] (\x-1, {4/\x}) ;
		
		\draw
		plot[domain=1:2] ( \x-1, {2-2/\x} );	
		
		\draw
		plot[domain=1:2] ( \x-1, {2/\x} )
		;

		
		\node [above left] at (0,0) {$0$};
		\node [below right] at (0,0) {$1$};
		
		\draw[->]  (0,0) -- (0,4.3) node [left] {$\mu$};
		\draw[->]  (0,0) -- (4.2,0) node [below] {$p$};

		\node [below] at (3.45,1.35) {$p=p_S(1+\mu)$};
		\node [right] at (0.35,0.5) {$p=\tfrac{2}{1+|\mu-1|}$};
		\node [above right,rotate=0] at (0.3,3) {$p=p_*(n,\mu)$};
		
		
		\node[left,opacity=0] at (0,0) {$\mu_*=\frac{4}{3}$};
		\node[left,opacity=0] at (0,0) {$n+3$};
		
		\end{scope}
		\end{tikzpicture}
	}
	\quad
	\subfloat[][Case $n\ge2$.]{
		\begin{tikzpicture}[scale=1.5, semithick]
		
		\begin{scope}[font=\footnotesize]
		
		\fill[wavecol]
		(0,0)
		-- plot [domain=1:1+2/3] ( \x*2-2, {(8/\x-2)/2} )
		-- plot [domain=1+2/3:1+sqrt(32)/4]
		( \x*2-2, {((-2*\x*\x+4*\x+2)/(\x*\x-\x))/2} )		
		-- cycle;

		\fill[heatcol]
		plot [domain=1+2/3:1] ( \x*2-2, {(8/\x-2)/2} )
		-- (0,6/2) -- (0,4.1) -- (4/3,4.1)
		-- cycle;
		
		\fill[mixedcol]
		plot [domain=1+2/3:1] ( \x*2-2, {(8/\x-2)/2} ) -- (0,6/2) -- (0,4.1) -- (4/3,4.1) -- cycle;
		
		
		\draw[loosely dotted] (0,1/2+9/5/2) node [left] {$\mu_*$} -- (2/3*2,1/2+9/5/2);
		
		\draw[loosely dotted] (2/3*2,0) node [below] {$p_F(n)$} -- (2/3*2,1/2+9/5/2);
		
		\draw[thick] (2/3*2,1/2+9/5/2) -- (2/3*2,4.1);
		
		\draw[thick]
		plot [domain=1+2/3:1+sqrt(32)/4]
		( \x*2-2, {((-2*\x*\x+4*\x+2)/(\x*\x-\x))/2} );	
		
		\draw[densely dash dot, thick]
		plot [domain=1:1+2/3] ( \x*2-2, {(8/\x-2)/2} );

		\node [left] at (0,6/2) {$n+3$};
		
		\node [below] at (1.41421*2, 0) {$p_S(n)$};

		
		\node at (0,0) {};
		\node [above left] at (0,0) {$0$};
		\node [below right] at (0,0) {$1$};
		
		\draw[->]  (0,0) -- (0,4.3) node [left] {$\mu$};
		\draw[->]  (0,0) -- (4.2,0) node [below] {$p$};

		\node [above right] at (1.75,0.65) {$p=p_S(n+\mu)$};
		\node [above] at (0.75,2.55) {$p=p_*(n,\mu)$};

		
		\node[left,opacity=0] at (0,0) {$\mu_*=\frac{4}{3}$};
		\node[left,opacity=0] at (0,0) {$n+3$};
		
		\end{scope}
		\end{tikzpicture}
	}
	\caption{Here we collect the results from Theorem \ref{cor2}. If $(p,\mu)$ is in the blue area, $T_\e \lesssim \e^{-2p(p-1)/\gamma_S(p,n+\mu)}$, hence the lifespan estimate is wave-like.
		If $(p,\mu)$ is in the purple area, $T_\e \lesssim \e^{-p(p-1)/\gamma_F(p,n)}$ and the lifespan estimate is of mixed-type.
		The dash-dotted line given by $p=p_*(n,\mu)$ highlights the improvement $T_\e \lesssim \sigma_0(\e)$.
		In the case $n=1$, if $(p,\mu)$ is in the red area, $T_\e \lesssim \e^{-(p-1)/\gamma_F(p,1+[\mu-1]_-)}$ and the lifespan estimate is heat-like.}
	\label{fig2}
\end{figure}
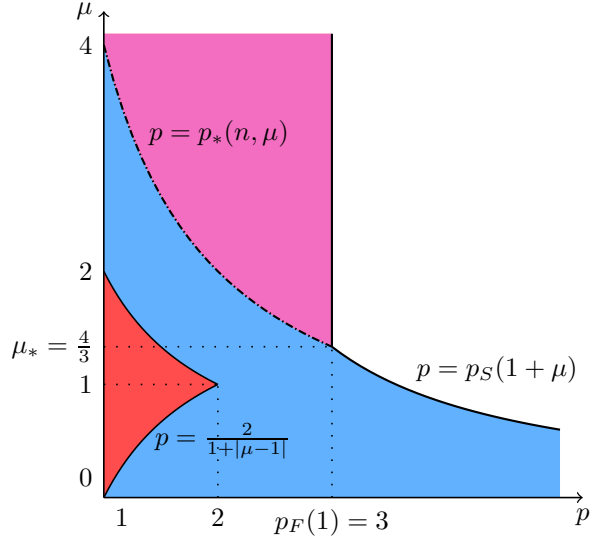
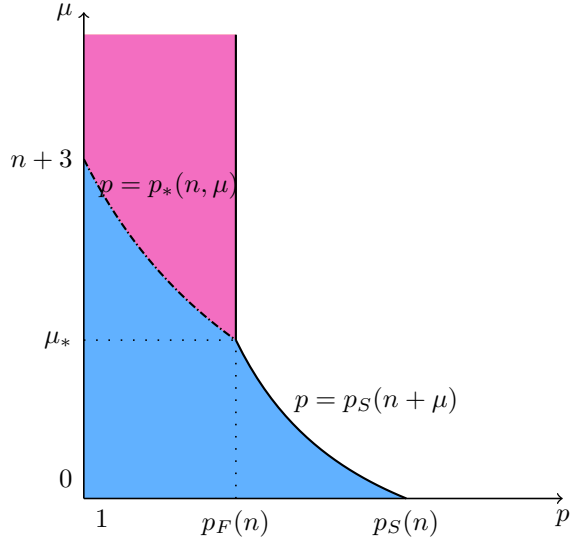


\begin{theorem}\label{thm:H=0}
	Let $\mu_1 \ge 0$, $\mu_2\in\R$, $\delta\ge0$ and $1< p < p_{\mu_1,\delta}(n)$, with $p_{\mu_1,\delta}(n)$ defined in \eqref{pch>0}.
	Assume that $f \in H^1(\R^n)$,
	$g \in L^2(\R^n)$ and $f >0$, $h=0$,
	with $h$ defined in \eqref{hpindata}.
	Suppose that $u$ is an energy solution of \eqref{eq:main_problem} on $[0,T)$ that satisfies
	\eqref{support}
	with some $R\ge1$.
	
	Then, there exists a constant $\e_4=\e_4(f,g,\mu_1,\mu_2,p, R)>0$
	such that the blow-up time $T_\e$ of problem \eqref{eq:main_problem}, for $0<\e\le\e_4$, has to satisfy:
	\begin{itemize}
		\item if $\sqrt{\delta} \le n-d_*(n+\mu_1)$, then
		\begin{equation*}
			T_{\e} \lesssim \e^{-2p(p-1)/\gamma_S(p,n+\mu_1)};
		\end{equation*}
		
		\item if $n-d_*(n+\mu_1) < \sqrt{\delta} < n+2$, then
		\begin{equation*}
		T_\e \lesssim
		\left\{
		\begin{aligned}
		& \e^{-2p(p-1)/\gamma_S(p,n+\mu_1)},
		&&\text{if $1<p< p_*$,}
		\\
		& \sigma_*(\e)
		&&\text{if $p=p_*$,}
		\\
		& \sigma(\e),
		&&\text{if $p_* < p < p_{\mu_1,\delta}(n)$,}
		\end{aligned}
		\right.
		\end{equation*}
		where $\sigma \equiv \sigma(\e)$ and $\sigma_* \equiv \sigma_*(\e)$ are the solutions respectively of
		\begin{gather*}
		\e^p \sigma^{\frac{\gamma_{F}(p,n+(\mu_1-1-\sqrt{\delta})/2)}{p-1}}
		\ln(1+\sigma)^{1-\sgn\delta} = 1,
		\\
		\e^p \sigma_*^{\frac{\gamma_{F}(p,n+(\mu_1-1-\sqrt{\delta})/2)}{p-1}}
		\ln(1+\sigma_*)^{2-\sgn\delta} = 1;
		\end{gather*}	
		
		\item if $\sqrt{\delta} \ge n+2$, then
		\begin{equation*}
			T_\e \lesssim \sigma(\e).
		\end{equation*}
	\end{itemize}

	Moreover, if $n=1$, $0\le \delta < 1$ and
	\begin{equation}\label{r_*}
	1< p < r_*(\mu_1,\delta) :=
	\left\{
	\begin{aligned}
	& 1+ 2 \, \frac{2-\sqrt{\delta}}{1+\mu_1+\sqrt{\delta}},
	&\quad&\text{if $\sqrt{\delta}<\theta$},
	\\
	& 1+ 2 \, \frac{2-\theta}{1+\mu_1+\theta} = \frac{2}{1+\theta},
	&\quad&\text{if $\sqrt{\delta} = \theta$},
	\\
	& \frac{2}{1+\sqrt{\delta}},
	&\quad&\text{if $\sqrt{\delta} > \theta$},
	\end{aligned}
	\right.
	\end{equation}
	with
	\begin{equation}\label{theta}
	\theta \equiv \theta(\mu_1) := 1 + \frac{\mu_1}{2} - \frac{1}{2} \sqrt{\mu_1^2 +16}
	\in (-1,1),
	\end{equation}
	then the estimate for the blow-up time $T_\e$ is improved by
	\begin{equation*}
	T_{\e} \lesssim \e^{-(p-1)/\gamma_F\left(p,(\mu_1+1+\sqrt{\delta})/2\right)}.
	\end{equation*}
\end{theorem}

\begin{remark}\label{rem:r*}
	In the $1$-dimensional case of Theorem\til\ref{thm:H=0}, one can check that $r_*<p_{\mu_1,\delta}(1)$ holds always, except when $\mu_1=3$ and $\delta=0$, since in this case $r_*=p_{3,0}(1)=p_S(4)=2$.
	About the relation between $p_*$ and $r_*$, we have that, for $0 \le {\delta}<1$, if $\sqrt{\delta} \lesseqgtr \theta$ then $p_* \lesseqgtr r_*$.
\end{remark}

\begin{remark}
	We conjecture that the estimates in the previous two theorems are indeed optimal, except in dimension $n=1$ for Theorem\til\ref{cor2} on the transition curve defined by
	\begin{equation*}
		p = \frac{2}{1+|\mu-1|} \quad \text{for $0\le \mu \le 2$,}
	\end{equation*}
	and for Theorem\til\ref{thm:H=0} on the transition surface
	\begin{equation*}
		p=r_*(\mu_1,\delta) \quad \text{for $0\le \delta \le 1$.}
	\end{equation*}
	
	Moreover, in the critical case we expect, due to the wave-like and mixed-like behaviours,
	\begin{equation*}
		T_\e \sim \exp(C \e^{-p(p-1)}),
	\end{equation*}
	except for $\sqrt{\delta}=n-d_*(n+\mu_1)$ and $p=p_{\mu_1,\delta}(n)$, where the lifespan should be different.
\end{remark}

\begin{remark}\label{rem:weaker-initial-data}
	The conditions \eqref{hpindata} on the initial data in Theorem~\ref{cor1} \& \ref{thm1} can be replaced by the less strong conditions
	\begin{align*}
		&\ints f(x) \ge 0,
		&&
		\ints h(x) > 0,
		\\
		&\ints f(x)\phi_1(x) \ge 0,
		&&
		\ints h(x)\phi_1(x) > 0,
	\end{align*}
	where the positive function $\phi_1(x)$ is defined later in \eqref{YZtest}.
	
	Similarity can be done for the initial conditions of Theorem~\ref{cor2} \& \ref{thm:H=0}, requiring
	\begin{align*}
		&\ints f(x) > 0,
		&&
		\ints h(x) = 0,
		\\
		&\ints f(x)\phi_1(x) > 0,
		&&\ints h(x)\phi_1(x) = 0.
	\end{align*}
	It will be clear from the proof of our theorems that these weaker hypothesis are sufficient.
\end{remark}


\subsection{Wave equation with scattering damping and negative mass}

Finally, in this subsection we want to continue the study of a problem examined by the authors in \cite{LST19.1,LST19.2}. In these two works, we considered the Cauchy problem for the wave equation with scattering damping and negative mass term, thus
\begin{equation}
\label{problem_negmass}
\left\{
\begin{aligned}
& w_{tt} - \Delta w + \frac{\nu_1}{(1+t)^\beta} w_t + \frac{\nu_2}{(1+t)^{\alpha+1}} w = |w|^p, \quad\text{in $\R^n\times(0,T)$}, \\
& w(x,0)=\e f(x), \quad w_t(x,0)=\e g(x), \quad x\in\R^n,
\end{aligned}
\right.
\end{equation}
where $\nu_1\ge0$, $\nu_2<0$, $\alpha\in\R$ and $\beta>1$.

In Subsection\til\ref{sub:damped} we already observed that, if the damping is of scattering type, the solution of the homogeneous damped wave equation \lq\lq scatters'' to the one of the wave equation. For the equation with power non-linearity, according to the results by Lai and Takamura \cite{LT18} and Wakasa and Yordanov \cite{WY19}, the solution again seems to be wave-like both in  the critical exponent and in the lifespan estimate.

In \cite{LST19.1}, the authors took in consideration \eqref{problem_negmass} with $\alpha>1$ and observed a double scattering phenomenon, in the sense that both the damping and the mass terms seem to be not effective. Hence, the solution behaves like that of the wave equation with power non-linearity $u_{tt}-\Delta u = |u|^p$. More precisely, supposing $f,g>0$ for the simplicity, we established the blow-up for $1<p<p_S(n)$ and the upper bound for the lifespan estimates:
\begin{equation*}
	T_\e \lesssim
	\left\{
	\begin{aligned}
	&\e^{-(p-1)/\gamma_F(p,n-1)}
	&& \text{if $n=1$ or $n=2,1<p<2$,}
	\\
	& a(\e)
	&&\text{if $n=p=2$,}
	&&
	\\
	& \e^{-2p(p-1)/\gamma_S(p,n)}
	&& \text{if $n=2,2<p<p_S(n)$ or if $n\ge3$,}
	\end{aligned}
	\right.
\end{equation*}
where $a\equiv a(\e)$ satisfies $\e^2 a^2 \log(1+a)=1$, although in the case $n=p=2$ more technical conditions were required.

In \cite{LST19.2}, the authors studied the case $\alpha<1$, discovering a new behaviour in the lifespan estimate. Indeed, we proved blow-up for every $p>1$ and the upper lifespan estimate
\begin{equation*}
	T_\e \lesssim \zeta(C\e),
\end{equation*}
where $\zeta\equiv \zeta(\o{\e})$ is the larger solution of the equation
\begin{equation*}
	\o{\e} \zeta^{\frac{\gamma_F(p,n-(1+\alpha)/4)}{(p-1)}}
	\exp\left( K \zeta^{\frac{1-\alpha}{2}}\right)=1,
	\quad
	\text{with $K = \frac{2 \sqrt{|\nu_2|}}{1-\alpha} \exp\left(\frac{\nu_1}{2(1-\beta)}\right)$.}
\end{equation*}
As observed in Remark\til2.1 of \cite{LST19.2}, a less sharp but more clear estimate for the lifespan in the case $\alpha<1$ is
\begin{equation*}
	T_\e \lesssim \left[\log\left( 1/\e\right)\right]^{2/(1-\alpha)}.
\end{equation*}
Hence, if the negative mass term with $\alpha>1$ seems to have no influence on the behaviour of the solution, on the contrary if $\alpha<1$ the negative mass term becomes extremely relevant, implying the blow-up for all $p>1$ and a lifespan estimate which is much shorter, compared to the ones introduced previously.

We come now to the case $\alpha=1$. This is particular and was not deepened in our previous works. Indeed in Subsection\til\ref{sub:negmass}, after introducing a multiplier to absorb the damping term, we will show that we can get blow-up results and lifespan estimates for this problem by reducing ourself to calculations similar to the ones we will perform to prove the results in the previous subsections. Roughly speaking, we will find out that \eqref{problem_negmass} with $\alpha=1$ has the same behaviour of \eqref{eq:main_problem} with $\mu_1=0$ and $\mu_2=\nu_2 e^{\nu_1/(1-\beta)}$.

Therefore, in the following we will consider the Cauchy problem
\begin{equation}
\label{problem_negmass_a=1}
\left\{
\begin{aligned}
& w_{tt} - \Delta w + \frac{\nu_1}{(1+t)^\beta} w_t + \frac{\nu_2}{(1+t)^{2}} w = |w|^p, \quad\text{in $\R^n\times(0,T)$}, \\
& w(x,0)=\e f(x), \quad w_t(x,0)=\e g(x), \quad x\in\R^n,
\end{aligned}
\right.
\end{equation}
where $\nu_1\ge0$, $\nu_2<0$ and $\beta>1$.

\begin{definition}
	We say that $u$ is an energy solution of \eqref{problem_negmass_a=1} over $[0,T)$ if
	\begin{equation*}
	w \in C([0,T), H^1(\R^n)) \cap C^1([0,T),L^2(\R^n)) \cap C((0,T), L^p_{loc}(\R^n))
	\end{equation*}
	satisfies $w(x,0)=\e f(x)$ in $H^1(\R^n)$, $w_t(x,0) = \e g(x)$ in $L^2(\R^n)$ and
	\begin{equation}\label{def:negmass}
	\begin{split}
	&\int_{\R^n}w_t(x,t)\phi(x,t)dx
	+\int_0^tds\int_{\R^n}\left\{-w_t(x,s)\phi_t(x,s)+\nabla w(x,s)\cdot\nabla\phi(x,s)\right\}dx \\
	&+\int_0^tds\int_{\R^n}\frac{\nu_1 }{(1+s)^{\beta}}w_t(x,s) \phi(x,s)dx + \int_0^t ds \ints \frac{\nu_2}{(1+s)^{2}}w(x,s)\phi(x,s)dx \\
	= &
	\int_{\R^n}\e g(x)\phi(x,0)dx
	+
	\int_0^tds\int_{\R^n}|w(x,s)|^p\phi(x,s)dx
	\end{split}
	\end{equation}
	with any test function $\phi \in C_0^\infty(\R^n \times [0,T))$ for $t\in[0,T)$.
	
\end{definition}

We have the following result. See Figure\til\ref{fig3} for a graphic representation of them.

\begin{theorem}\label{thm_negmass}
	Fix $\nu_1\ge0$, $\nu_2<0$, $\beta>1$. Define
	\begin{equation*}
	\delta := 1-4 \nu_2 e^{{\nu_1}/{(1-\beta)}} >1,
	\qquad
	d_*(n) := \frac{1}{2} \left(-1- n + \sqrt{n^2+10n-7} \right) \in [0,2)
	\end{equation*}
	and let
	$1< p < p_{\delta}(n) $, with
	\begin{equation*}
	\begin{aligned}
	p_{\delta}(n) &=
	\max\left\{
	p_F\left(n-\frac{1+\sqrt{\delta}}{2}\right), \,
	p_S\left(n\right) \right\}
	\\
	&=
	\left\{
	\begin{aligned}
	& p_S\left(n\right)
	&\quad&\text{if $n \ge 2$, $\sqrt{\delta} \le n-d_*(n)$,}
	\\
	& p_F\left(n-\frac{1+\sqrt{\delta}}{2}\right)
	&\quad&\text{if $n \ge 2$, $n-d_*(n) < \sqrt{\delta}<2n-1$,}
	\\
	& +\infty
	&\quad&\text{if $n = 1$ or if $n\ge2$, $\sqrt{\delta}\ge 2n-1$.}
	\end{aligned}
	\right.
	\end{aligned}
	\end{equation*}
	Assume that $f \in H^1(\R^n)$,
	$g \in L^2(\R^n)$ are non-negative and not both vanishing.
	Suppose that $w$ is an energy solution of \eqref{problem_negmass_a=1} on $[0,T)$ that, for some $R\ge1$, satisfies
	\begin{equation*}
	\supp w \subset\{(x,t)\in\R^n\times[0,\infty) \colon |x|\le t+R\}.
	\end{equation*}

	Then, there exists a constant $\e_5=\e_5(f,g,\beta,\nu_1,\nu_2,n,p, R)>0$
	such that the blow-up time $T_\e$ of problem \eqref{problem_negmass_a=1}, for $0<\e\le\e_5$, has to satisfy:
	\begin{itemize}
		\item if $\sqrt{\delta} \le n-2$, then
		\begin{equation*}
		T_{\e} \lesssim \e^{-2p(p-1)/\gamma_S(p,n)};
		\end{equation*}
		
		\item if $n-2 < \sqrt{\delta} < n-d_*(n)$, then
		\begin{equation*}
		T_\e \lesssim
		\left\{
		\begin{aligned}
		&\e^{-(p-1)/\gamma_F(p,n-(1+\sqrt{\delta})/2)}
		,
		&&\text{if $1<p\le \frac{2}{n-\sqrt{\delta}}$,}
		\\
		&\e^{-2p(p-1)/\gamma_S(p,n)},
		&&\text{if $\frac{2}{n-\sqrt{\delta}} < p < p_{\delta}(n)$;}
		\end{aligned}
		\right.
		\end{equation*}
		
		\item if $\sqrt{\delta} \ge n - d_*(n)$, then
		\begin{equation*}
		T_{\e} \lesssim \e^{-(p-1)/\gamma_F(p,n-(1+\sqrt{\delta})/2)}
		=
		\e^{-\left[ 2/(p-1) - n + (1+\sqrt{\delta})/2 \right]^{-1}}.
		\end{equation*}
	\end{itemize}
\end{theorem}

\begin{figure}
	\centering
	\subfloat[][Case $n=1$.]{
		\begin{tikzpicture}[scale=1.5, semithick]
		
		\begin{scope}[font=\footnotesize]
		
		\fill[heatcol]
		plot[domain=5.5:2] ( \x-1, {(1-2/\x)} )
		-- (1,0) -- (0,0) -- (0,3) -- (4.5,3)
		-- cycle;
		
		\fill[wavecol]
		(4.5,0) --
		plot[domain=2:5.5] ( \x-1, {(1-2/\x)} )
		-- cycle;
		
		
		\draw[loosely dotted] (0,1) node [left] {$1$} -- (4.5,1);
		
		\draw
		plot[domain=2:5.5] ( \x-1, {(1-2/\x)} );
		
		\node[below] at (1,0) {$2$};
		
		
		\node [above left] at (0,0) {$0$};
		\node [below right] at (0,0) {$1$};
		
		\draw[->]  (0,0) -- (0,3.2) node [left] {$\sqrt{\delta}$};
		\draw[->]  (0,0) -- (4.7,0) node [below] {$p$};

		\node [below left] at (1.8,0.7) {$p=\tfrac{2}{1-\sqrt{\delta}}$};
		
		\node[left, opacity=0] at (0,0) {$2-d_*(2)$};
		\node[left, opacity=0] at (0,0) {$n-d_*(n)$};
		
		\end{scope}
		\end{tikzpicture}
	}
	\quad
	\subfloat[][Case $n=2$.]{
		\begin{tikzpicture}[scale=1.5, semithick]
		
		\begin{scope}[font=\footnotesize]
		
		\fill[wavecol]
		(2.5615528,0) --
		plot[domain=1:3.5615528] ( \x-1, {(2-2/\x)} )
		-- cycle;
		
		\fill[heatcol]
		(4.5,3.5) -- (0,3.5) --
		plot[domain=1:3.5615528] ( \x-1, {(2-2/\x)} )
		-- plot [domain=3.5615528:5.5] ( \x-1, {3-4/(\x-1)} )
		-- cycle;
		
		
		\draw[loosely dotted] (0,1.438447) node [left] {$2-d_*(2)$} -- (3.5615528-1,1.438447);
		
		\draw[loosely dotted] (0,3) node [left] {$3$} -- (4.5,3); 
		
		\draw[thick] (3.5615528-1,0) node[below] {$p_S(2)$} -- (3.5615528-1,1.438447);
		
		\draw[thick]
		plot [domain=3.5615528:5.5] ( \x-1, {3-4/(\x-1)} );	
		
		\draw
		plot[domain=1:3.5615528] ( \x-1, {(2-2/\x)} );


		
		\node at (0,0) {};
		\node [above left] at (0,0) {$0$};
		\node [below right] at (0,0) {$1$};
		
		\draw[->]  (0,0) -- (0,3.7) node [left] {$\sqrt{\delta}$};
		\draw[->]  (0,0) -- (4.7,0) node [below] {$p$};

		\node [below] at (3,2.5) {$p=p_F\left(2-\tfrac{1+\sqrt{\delta}}{2}\right)$};
		\node [below left] at (1.65,0.7) {$p=\tfrac{2}{2-\sqrt{\delta}}$};

		\node[left, opacity=0] at (0,0) {$2-d_*(2)$};
		\node[left, opacity=0] at (0,0) {$n-d_*(n)$};
		
		\end{scope}
		\end{tikzpicture}
	}
	\quad
	\subfloat[][Case $n\ge3$.]{
		\begin{tikzpicture}[scale=1.5, semithick]
		
		\begin{scope}[font=\footnotesize]
		
		\fill[wavecol]
		(2.41421356-1,0) -- (0,0) --
		plot[domain=1:2.41421356] ( \x-1, {(3-2/\x)/1.5} )
		-- cycle;
		
		\fill[heatcol]
		(4.5,3.5) -- (0,3.5) -- (0,1/1.5) --
		plot[domain=1:2.41421356] ( \x-1, {(3-2/\x)/1.5} )
		-- plot [domain=2.41421356:5.5] ( \x-1, {(5-4/(\x-1))/1.5} )
		-- cycle;
		
		
		\draw[loosely dotted] (0,2.171572875/1.5) node [left] {$n-d_*(n)$} -- (2.41421356-1,2.171572875/1.5);
		
		\draw[loosely dotted] (0,1/1.5) node [left] {$n-2$} -- (2.41421356-1,1/1.5);
		
		\draw[loosely dotted] (0,5/1.5) node [left] {$2n-1$} -- (4.5,5/1.5); 
		
		\draw[thick] (2.41421356-1,0) node[below] {$p_S(n)$} -- (2.41421356-1,2.171572875/1.5);
		
		\draw[thick]
		plot [domain=2.41421356:5.5] ( \x-1, {(5-4/(\x-1))/1.5} );	
		
		\draw
		plot[domain=1:2.41421356] ( \x-1, {(3-2/\x)/1.5} );


		
		\node at (0,0) {};
		\node [above left] at (0,0) {$0$};
		\node [below right] at (0,0) {$1$};
		
		\draw[->]  (0,0) -- (0,3.7) node [left] {$\sqrt{\delta}$};
		\draw[->]  (0,0) -- (4.7,0) node [below] {$p$};

		\node[below right] at (0.25,1.15) {$p=\tfrac{2}{n-\sqrt{\delta}}$};
		\node[below right] at (2.5,2.4) {$p=p_F\left(n-\tfrac{1+\sqrt{\delta}}{2}\right)$};

		\node[left, opacity=0] at (0,0) {$2-d_*(2)$};
		\node[left, opacity=0] at (0,0) {$n-d_*(n)$};
		
		\end{scope}
		\end{tikzpicture}
	}
	\caption{Here we collect the results from Theorem~\ref{thm_negmass}. If $(p,\sqrt{\delta})$ is in the blue area, $T_{\e} \lesssim \e^{-2p(p-1)/\gamma_S(p,n)}$, hence the lifespan estimate is wave-like.
		Otherwise, if $(p,\sqrt{\delta})$ is in the red area, $T_{\e} \lesssim \e^{-(p-1)/\gamma_F(p,n-(1+\sqrt{\delta})/2)}$ and the lifespan is heat-like.
		Note that this figure represents also the results of Theorem~\ref{thm1} for the case $\mu_1=0$, $\mu_2 \le 1/4$.}
	\label{fig3}
\end{figure}
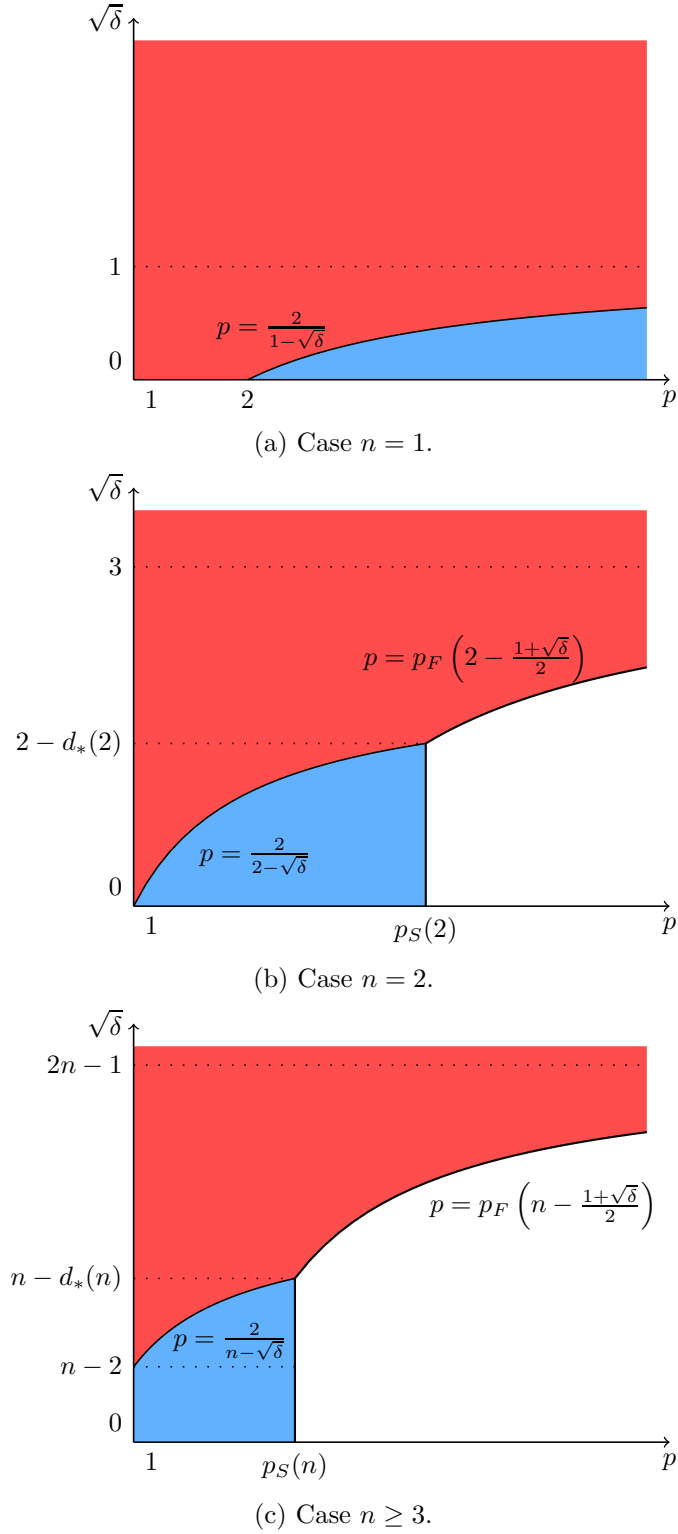

\begin{remark}
	As a direct consequence of Remark\til\ref{rem:conj1} \& \ref{rem:conj2}, we expect that $p_{\delta}(n)$ is the critical exponent and that the lifespan estimates presented in Theorem\til\ref{thm_negmass} are optimal, except on the \emph{transition curve} (in the $(p,\delta)$-plane) defined by
	\begin{equation*}
	p= \frac{2}{n-\sqrt{\delta}}
	\quad\text{for $n-2<\sqrt{\delta}<n-d_*(n)$ and $1<p \le p_{\delta}(n)$,}
	\end{equation*}
	on which we presume a logarithmic gain can appear.
	
 	Moreover, we expect that, if $p=p_{\delta}(n)$, then
	\begin{equation*}
	T_\e \sim 
	\left\{
	\begin{aligned}
	& \exp\left( C \e^{-p(p-1)}\right)
	&&\text{if $n\ge2$, $\sqrt{\delta}< n-d_*(n)$, 
	}
	\\
	& \exp\left( C \e^{-(p-1)}\right)
	&&\text{if $n\ge2$, $n-d_*(n) < \sqrt{\delta}<2n-1$, 
	}
	\end{aligned}
	\right.
	\end{equation*}
	for some constant $C>0$. If $\sqrt{\delta}=n-d_*(n)$ and $p=p_{\delta}(n)$, we presume a lifespan estimate of different kind.
\end{remark}



\section{Kato's type lemma}\label{sec:kato}
The principal ingredient we will employ in the demonstration of our theorems is the following Kato's type lemma. Although this tool is well known and used in the literature, here we will reformulate it in such a way, in the following sections, we can directly apply it to obtain not only the condition to find the possible critical exponent, but also the upper lifespan estimate. We will prove it using the so called iteration argument.

\begin{lemma}\label{kato}
	Let $p>1$, $a,b \in\R$ satisfy
	\begin{equation*}
		\gamma := 2[(p-1)a - b + 2] >0.
	\end{equation*}
	Assume that $F \in C([0,T))$ satisfies, for $t \ge T_0$,
	\begin{gather}
		\label{hp1}
		F(t) \ge EA t^a \left[\ln(1+t)\right]^c
		,
		\\
		\label{hp2}
		F(t) \ge
		B
		\int_{T_0}^{t}ds \int_{T_0}^{s} r^{-b} F(r)^p dr
		,
	\end{gather}
	where $c,T_0 \ge0$ and $E,A,B >0$.
	Then, for $\widetilde{T} \ge T_0$ we have that
	\begin{equation*}
		T<C\widetilde{T}
	\end{equation*}
	holds, assumed that
	\begin{equation}
	\label{hp:ttilde}
		E \widetilde{T}^{\frac{\gamma}{2(p-1)}}
		\left[\ln(1+\widetilde{T})\right]^c
		=1,
	\end{equation}
	where $C$ is a constant independent of $E$.
\end{lemma}
\begin{proof}
	Let $\widetilde{T}$ be as in the statement of the lemma and start with the ansatz
	\begin{equation}
	\label{ansatz}
		F(t) \ge D_j \left[\ln(1+\widetilde{T})\right]^{c_j}
		t^{-b_j} (t-\widetilde{T})^{a_j}
		\quad\text{for $t \ge \widetilde{T}$, \quad $j=1,2,3,\dots$}
	\end{equation}
	where $D_j, a_j, b_j, c_j$ are positive constants to be determined later. Due to hypothesis \eqref{hp1}, observe that \eqref{ansatz} is true for $j=1$ with
	\begin{equation}
	\label{basis}
		D_1 = EA, \quad a_1 = [a]_+, \quad b_1=[a]_-, \quad c_1 = c,
	\end{equation}
	where $[x]_\pm := (|x| \pm x)/2$.
	Plugging \eqref{ansatz} into \eqref{hp2}, we get
	\begin{equation*}
		\begin{split}
		F(t) &\ge D_j^p B
		\int_{\widetilde{T}}^{t}ds \int_{\widetilde{T}}^{s}
		\left[\ln(1+\widetilde{T})\right]^{pc_j}
		r^{-b-pb_j} (r-\widetilde{T})^{pa_j}
		dr
		\\
		&\ge \frac{D_j^p B}{(pa_j+ [b]_- +2)^2} \left[\ln(1+\widetilde{T})\right]^{pc_j} t^{-pb_j-[b]_+ }(t-\widetilde{T})^{pa_j+ [b]_- +2}
		\quad\text{for $t\ge\widetilde{T}$,}
		\end{split}
	\end{equation*}
	and then we can define the sequences $\{D_j\}_{j\in\N}$, $\{a_j\}_{j\in\N}$, $\{b_j\}_{j\in\N}$, $\{c_j\}_{j\in\N}$ by
	\begin{equation*}
	D_{j+1}=\frac{
	D_{j}^p B}{(pa_{j}+ [b]_- +2)^2},
	\quad
	a_{j+1}=pa_j+[b]_- +2,
	\quad
	b_{j+1}=pb_j+[b]_+,
	\quad
	c_{j+1}=pc_j,
	\end{equation*}
	to establish \eqref{ansatz} with $j$ replaced by $j+1$.
	It follows from the previous relations and \eqref{basis} that for $j\ge1$
	\begin{equation*}
	a_j= p^{j-1} \left([a]_+ +\frac{[b]_- +2}{p-1}\right) - \frac{[b]_- +2}{p-1},
	\quad
	b_j= p^{j-1} \left( [a]_- + \frac{[b]_+}{p-1} \right) - \frac{[b]_+}{p-1},
	\quad
	c_j= p^{j-1}c.
	\end{equation*}
	In particular, we obtain that
	\begin{equation}
	\label{24a}
		pa_j+[b]_- +2
		=a_{j+1}
		\le p^j\left([a]_+ +\frac{[b]_- +2}{p-1}\right)
		\Longrightarrow
		D_{j+1}\ge \widetilde{C} p^{-2j} D_j^p,
	\end{equation}
	where $\widetilde{C}:={B}/{\left\{ [a]_+ +{([b]_- +2)}/{(p-1)}\right\}^2} >0$.
	From \eqref{24a} and $D_1=EA$, by an inductive argument we infer that
	\[
	D_j\ge\exp\left\{p^{j-1}\left[\ln(EA) -S_j \right]\right\},
	\]
	where
	\[
	S_j:=\sum_{k=1}^{j-1}\frac{2k\ln p-\ln \widetilde{C}}{p^k}.
	\]
	Since $\sum_{k=0}^{\infty} x^k = 1/(1-x)$ and $\sum_{k=1}^{\infty} kx^k= x/(1-x)^2$ when $|x|<1$, we obtain
	\[
	S_\infty := \lim_{j \to +\infty} S_j = \ln\{ \widetilde{C}^{p/(1-p)} p^{2p/(1-p)^2}\}.
	\]
	Moreover $S_j$ is a sequence definitively increasing. Hence we obtain that
	\begin{equation*}
		D_j \ge (EA e^{-S_\infty})^{p^{j-1}}
	\end{equation*}
	for $j$ sufficiently large. Let us turning back to \eqref{ansatz} and let $C>1$ a constant to be determined later. Supposing $t\ge C\widetilde{T}$, so that in particular $t-\widetilde{T}\ge (1-1/C)t$, and considering \eqref{hp:ttilde}, we have
	\begin{equation}
	\label{27}
	\begin{aligned}
	F(t)&\ge
	t^{\frac{[b]_+}{p-1}} (t-\widetilde{T})^{-\frac{[b]_- +2}{p-1}}
	\left\{
	EA e^{-S_\infty} \left[\ln(1+\widetilde{T})\right]^c t^{-[a]_- - \frac{[b]_+}{p-1}} (t-\widetilde{T})^{[a]_+ + \frac{[b]_- +2}{p-1}}
	\right\}^{p^{j-1}} \\
	&\ge
	t^{\frac{[b]_+}{p-1}} (t-\widetilde{T})^{-\frac{[b]_- +2}{p-1}}
	\left\{
	EA e^{-S_\infty} \left(1-\frac{1}{C}\right)^{[a]_+ + \frac{[b]_- +2}{p-1}} \left[\ln(1+\widetilde{T})\right]^c t^{\frac{\gamma}{2(p-1)}}
	\right\}^{p^{j-1}} \\
	&\ge
	t^{\frac{[b]_+}{p-1}} (t-\widetilde{T})^{-\frac{[b]_- +2}{p-1}}
	J^{p^{j-1}}
	\end{aligned}
	\end{equation}
	with
	$
		J := A e^{-S_\infty} \left(1-1/C\right)^{[a]_+ + \frac{[b]_- +2}{p-1}}
		C^{\frac{\gamma}{2(p-1)}}.
	$
	Since $\gamma>0$, we can choose $C>1$ big enough, such that $J>1$. Letting $j\to+\infty$ in \eqref{27}, we get $F(t)\rightarrow +\infty$. Then, $T<C\widetilde{T}$ as claimed.
\end{proof}

\begin{remark}
	We can observe that the previous lemma is still true if in \eqref{hp2} an arbitrary number of integrals appear, more precisely if we replace \eqref{hp2} with
	\begin{equation*}
		F(t) \ge
		B
		\int_{T_0}^{t} dt_1 \int_{T_0}^{t_1} dt_2 \cdots \int_{T_0}^{t_{k-1}} t_k^{-b} F(t_k)^p dt_k
		\quad\text{for $t \ge T_0$},
	\end{equation*}
	and $\gamma$ with $\gamma_k := 2[(p-1)a - b + k] $, where $k\ge1$ is an integer.
\end{remark}


\section{Proof for Theorems}\label{sec:proof}

We come now to the demonstration for Theorems\til\ref{thm1} \& \ref{thm:H=0}. In the next two subsections, we will prove some key inequalities which will be employed in the machinery of the Kato's type lemma. Applying the latter, we will find a couple of results, which will be compared in Subsection\til\ref{sub:comparison} to find the claimed ones. The proof of Theorems\til\ref{cor1} \& \ref{cor2} are clearly omitted, since they are corollaries of Theorems\til\ref{thm1} \& \ref{thm:H=0} respectively, just setting the mass equal to zero. In the end, we will sketch the proof for Theorem\til\ref{thm_negmass} in Subsection\til\ref{sub:negmass}.


\subsection{Key estimates}\label{sec:F0}

Let us define the functional
\begin{equation*}
F_0(t) := \int_{\R^n} u(x,t)dx.
\end{equation*}
Choosing the test function $\phi=\phi(x,s)$ in \eqref{eq:energy_solution} to satisfy
\begin{equation}\label{eq:testfunction}
	\text{$\phi\equiv 1$ in $\{(x,s)\in \R^n\times[0,t]:|x|\le s+R\}$,}
\end{equation}
we get
\begin{equation*}
\begin{split}
&\int_{\R^n}u_t(x,t)dx-\int_{\R^n}u_t(x,0)dx
\\
&+\int_0^tds\int_{\R^n}\frac{\mu_1}{1+s}u_t(x,s)dx
+ \int_0^t \int_{\R^n} \frac{\mu_2}{(1+s)^2} u(x,s) dx 
\\
=&\, \int_0^tds\int_{\R^n}|u(x,s)|^pdx,
\end{split}
\end{equation*}
which yields, by taking derivative with respect to $t$,
\begin{equation}
\label{2}
F_0''(t)+\frac{\mu_1}{1+t}F_0'(t)
+ \frac{\mu_2}{(1+t)^{2}}F_0(t) = \int_{\R^n}|u(x,t)|^pdx.
\end{equation}
Setting
\begin{equation*}
\label{def:kl}
\l:= 1+\sqrt{\delta} >0,
\quad
\k:= \frac{\mu_1-1-\sqrt{\delta}}{2},
\quad
	\delta := (\mu_1-1)^2-4\mu_2,
\end{equation*}
we obtain that \eqref{2} is equivalent to
\begin{equation*}
\label{3}
\frac{d}{dt} \left\{ (1+t)^{\l}
\frac{d}{dt} \left[(1+t)^{\k} F_0(t) \right]\right\}
= (1+t)^{\k+\l}\int_{\R^n}|u(x,t)|^p dx.
\end{equation*}
Integrating twice the above equality over $[0, t]$, we get
\begin{equation}
\label{4}
\begin{split}
F_0(t) = L(t) + M(t)
,
\end{split}
\end{equation}
where
\begin{gather*}
	L(t) := F_0(0)(1+t)^{-\k} + [\k F_0(0) + F'_0(0)] (1+t)^{-\k} \int_{0}^{t}(1+s)^{-\l} ds, \\
	M(t) := (1+t)^{-\k}
	\int_{0}^{t}(1+s)^{-\l} ds
	\int_{0}^{s}(1+r)^{\k+\l}dr
	\int_{\R^n}|u(x,r)|^p dx \ge 0 .
\end{gather*}

Define the functional \[\F(t):=(1+t)^{\kappa+\lambda}F_0(t)\]
and observe that $F_0$ and $\F$ implies the same blow-up results, so we will study the latter functional.
Since
\begin{equation*}
	\int_{\R^n} f(x)dx \ge0,
	\quad
	H_0 := \int_{\R^n} h(x)dx \ge0,
\end{equation*}
and they are not both equal to zero, we want to prove that exists a time $T_0>0$, independent of $\e$, such that, for $t \ge T_0$, the following estimates hold:
\begin{gather}
\label{key1}
	\F(t) \gtrsim \int_{T_0}^{t}ds \int_{T_0}^{s} r^{-(n+\k+\l)(p-1)} \F(r)^p dr,
\\
\label{key3}
\F(t) \gtrsim \e
\left\{
\begin{aligned}
& t
&\quad&\text{if $H_0=0$,}
\\
& t^{\l} \ln^{1-\sgn\delta}(1+t)
&\quad&\text{if $H_0>0$,}
\end{aligned}
\right.
\\
\label{key2}
\F(t) \gtrsim \e^p
\left\{
\begin{aligned}
& t^{\k+\l-(n+\mu_1-1)\frac{p}{2}+n+1}
&\quad&\text{if $\k-(n+\mu_1-1)\frac{p}{2}+n+1>0$,}
\\
& t^\l \ln^{2-\sgn\delta}(1+t)
&\quad&\text{if $\k-(n+\mu_1-1)\frac{p}{2}+n+1=0$,}
\\
& t^\l \ln^{1-\sgn\delta}(1+t)
&\quad&\text{if $\k-(n+\mu_1-1)\frac{p}{2}+n+1<0$.}
\end{aligned}
\right.
\end{gather}

	Thanks to the H\"older inequality and using the compact support of the solution \eqref{support}, we have
	\begin{equation}
	\label{intu1}
	\int_{\R^n}|u(x,t)|^pdx
	\gtrsim
	t^{-n(p-1)}|F_0(t)|^p
	= (1+t)^{-n(p-1)-(\k+\l)p}\F(t)^p
	\end{equation}
	for $t\gtrsim 1$.
	Considering $L$ and recalling the definition \eqref{hpindata} of $H_0$
	we obtain
	\begin{equation*}
	L(t) =
	\left\{
	\begin{aligned}
	&(1+t)^{-\k} [F_0(0)+ \e H_0 \ln(1+t)]
	&\quad&\text{if $\delta=0$,}
	\\
	&\frac{(1+t)^{-\k}}{\sqrt{\delta}} \left\{ \e H_0 + [\sqrt{\delta}F_0(0) - \e H_0](1+t)^{-\sqrt{\delta}}\right\}
	&\quad&\text{if $\delta>0$.}
	\end{aligned}
	\right.
	\end{equation*}
	So, from the condition on the initial data we get, for $t \gtrsim 1$ sufficiently large, that
	\begin{equation}
	\label{estL}
	L(t) \gtrsim \e
	\left\{
	\begin{aligned}
	&t^{-\k-\sqrt{\delta}}
	&\quad&\text{if $H_0=0$,}
	\\
	&t^{-\k}
	&\quad&\text{if $H_0>0$, $\delta>0$,}
	\\
	&t^{-\k} \ln(1+t)
	&\quad&\text{if $H_0>0$, $\delta=0$,}
	\end{aligned}
	\right.
	\end{equation}
	and in particular the positiveness of $L$ for large time. Neglecting $L$ from \eqref{4}, inserting \eqref{intu1} and recalling that $\l >0$, we get \eqref{key1}. Instead, inserting \eqref{estL} in \eqref{4} and neglecting $M$, we reach \eqref{key3}.
		
	Finally, we will prove \eqref{key2} in the next section.

\subsection{Weighted functional}\label{sec:F1}
Let us introduce
\begin{equation*}
F_1(t):=\int_{\R^n}u(x,t)\psi_1(x,t)dx,
\end{equation*}
where $\psi_1$ is the test function presented by Yordanov and Zhang in \cite{YZ06},
\begin{equation}\label{YZtest}
\psi_1(x,t):=e^{-t}\phi_1(x),
\qquad
\phi_1(x):=
\begin{cases}
\d\int_{S^{n-1}}e^{x\cdot\omega}dS_\omega & \text{for $n\ge2$},\\
e^x+e^{-x} & \text{for $n=1$},
\end{cases}
\end{equation}
which satisfies the following inequality (equation (2.5) in \cite{YZ06}):
	\begin{equation}
	\label{YZ}
	\int_{|x|\leq t+R}
	\psi_1(x,t) ^{\frac{p}{p-1}}dx
	\lesssim (1+t)^{(n-1) \left\{1-\frac{p}{2(p-1)} \right\}}.
	\end{equation}
%
We want to establish the lower bound for $F_1$. From the definition of energy solution \eqref{eq:energy_solution}, we have that
\begin{equation*}
	\begin{split}
	&\frac{d}{dt}\int_{\R^n}u_t(x,t)\phi(x,t)dx
	\\
	&-\int_{\R^n} u_t(x,t)\phi_t(x,t)dx
	-\int_{\R^n} u(x,t)\Delta\phi(x,t) dx\\
	&+\int_{\R^n}\frac{\mu_1 }{1+t}u_t(x,t)\phi(x,t)dx
	+ \int_{\R^n} \frac{\mu_2}{(1+t)^{2}}u(x,t)\phi(x,t)dx\\
	=&\ \int_{\R^n}|u(x,t)|^p\phi(x,t)dx.
	\end{split}
\end{equation*}
Integrating the above inequality over $[0,t]$, and in particular using integration by parts on the second term in the first line and on the first term in the second line, we get
\begin{equation}
\label{eqforF_1}
\begin{split}
&\int_{\R^n}u_t(x,t)\phi(x,t)dx
-\e\int_{\R^n}g(x)\phi(x,0)dx\\
&-\int_{\R^n} u(x,t)\phi_t(x,t)dx + \e \int_{\R^n} f(x)\phi_t(x,0)dx
\\
& + \int_0^t ds \int_{\R^n} u(x,s)\phi_{tt}(x,s)dx - \int_0^tds\int_{\R^n}
u(x,s)\Delta\phi(x,s) dx \\
&+ \int_{\R^n} \frac{\mu_1}{1+t}u(x,t)\phi(x,t)dx - \e \mu_1 \int_{\R^n} f(x)\phi(x,0)dx \\
&+ \int_0^t ds \int_{\R^n} u(x,s) \frac{\mu_1}{(1+s)^2}\phi(x,s)dx
- \int_0^t ds \int_{\R^n} u(x,s) \frac{\mu_1}{1+s}\phi_t(x,s)dx\\
&+ \int_0^tds\int_{\R^n} \frac{\mu_2}{(1+s)^{2}}u(x,s)\phi(x,s)dx
\\
=&\ \int_0^tds\int_{\R^n} |u(x,s)|^p\phi(x,s)dx.
\end{split}
\end{equation}
Setting
\[
\phi(x,t)=\psi_1(x,t)=e^{-t}\phi_1(x)
\quad\text{on $\supp u$},
\]
then we have
\[
\phi_t=-\phi,\quad \phi_{tt}=\Delta\phi \quad\text{on $\supp u$}.
\]
Hence we obtain from \eqref{eqforF_1}
\[
\begin{split}
&F_1'(t)+2F_1(t)
+ \frac{\mu_1}{1+t} F_1(t)
+ \int_0^t \left\{\frac{\mu_1}{1+s} + \frac{\mu_1+\mu_2}{(1+s)^2} \right\} F_1(s)ds \\
=& \ \e \int_{\R^n}\left\{ (1+\mu_1)f(x)+g(x)\right\}\phi_1(x)dx +\int_0^tds\int_{\R^n}|u(x,s)|^p \phi(x,s) dx,
\end{split}
\]
from which, after a derivation,
\begin{equation}
\label{6}
	F''_1(t) + \left(2+\frac{\mu_1}{1+t}\right) F'_1(t) + \left(\frac{\mu_1}{1+t} + \frac{\mu_2}{(1+t)^2}\right) F_1(t)
	= \int_{\R^n}|u(x,t)|^p \phi(x,t) dx
\end{equation}
Let us define the multiplier
\begin{equation*}
	m(t) := e^{t} (1+t)^{\frac{\mu_1-1}{2}} >0.
\end{equation*}
Then, multiplying equation \eqref{6} by $m(t)$, using for convenience the change of variables $z := 1+t$ and denoting
\begin{equation}
\label{def:B}
\B(z) := m(t)F_1(t),
\end{equation}
we obtain that $\B$ satisfies the nonlinear modified Bessel's equation
\begin{equation}
\label{kummer}
z^2 \frac{d^2\B}{dz^2} (z)
+ z \frac{d\B}{dz} (z)
- \left( z^2 +\frac{\delta}{4} \right) \B(z)
= N(z)
\end{equation}
with initial data
\begin{equation}
\label{kummer_data}
	\B(1) = \e \int_{\R^n} f(x)\phi_1(x)dx ,
	\quad
	\frac{d\B}{dz}(1) =
	\e \int_{\R^n} \left\{ \frac{\mu_1-1}{2} f(x)+g(x) \right\}\phi_1(x)dx ,
\end{equation}
and where
\begin{equation*}
	N(z) := z^2 m\left(z-1\right)\int_{\R^n}|u(x,z-1)|^p \phi(x,z-1) dx \ge 0.
\end{equation*}
Now we want to estimate $\B$.

\textit{Homogeneous problem.} Let us firstly consider the homogeneous Cauchy problem
\begin{equation*}
\left\{
\begin{aligned}
& z^2\frac{d^2\B_0}{dz^2} (z)
+ z \frac{d\B_0}{dz} (z)
- \left( z^2 +\frac{\delta}{4} \right) \B_0(z)
= 0,
\quad\text{$z \ge 1$}, \\
& \B_0(1) = \B(1),
\quad
\frac{d\B_0}{dz}(1) =
\frac{d\B}{dz}(1).
\end{aligned}
\right.
\end{equation*}
The fundamental solutions are the modified Bessel's functions $B^+_{\sqrt{\delta}/2}(z) := I_{\sqrt{\delta}/2}(z)$ and $B^-_{\sqrt{\delta}/2}(z) :=K_{\sqrt{\delta}/2}(z)$. Then we have
\begin{equation*}
	\B_0(z) =\e c_+ B^+_{\sqrt{\delta}/2}(z) + \e c_- B^-_{\sqrt{\delta}/2}(z),
\end{equation*}
where, thanks to equations (9.6.15) and (9.6.26) from Chapter~9 in~\cite{AS}, it holds
\begin{align*}
	c_\pm &= \pm \e^{-1} \left\{ \frac{d\B_0}{dz}(1) - \frac{\sqrt{\delta}}{2}\B_0(1) \right\} B^\mp_{\sqrt{\delta}/2}(1)
	+ \e^{-1} \B_0(1) B^\mp_{1+\sqrt{\delta}/2}(1) \\
	&= \pm B^\mp_{\sqrt{\delta}/2}(1)
	\int_{\R^n} h(x)\phi_1(x) dx
	+ \left[ \mp \sqrt{\delta} B^\mp_{\sqrt{\delta}/2}(1)
	+ B^\mp_{1+\sqrt{\delta}/2}(1) \right]
	\int_{\R^n} f(x) \phi_1(x) dx
	\\&=
	\left\{
	\begin{aligned}
	&\pm B^\mp_{0}(1)
	\int_{\R^n} h(x)\phi_1(x) dx
	+ B^\mp_{1}(1)
	\int_{\R^n} f(x) \phi_1(x) dx
	&\quad&\text{if $\delta=0$,}
	\\
	&\pm B^\mp_{\sqrt{\delta}/2}(1)
	\int_{\R^n} h(x)\phi_1(x) dx
	+ B^\mp_{-1+\sqrt{\delta}/2}(1)
	\int_{\R^n} f(x) \phi_1(x) dx
	&\quad&\text{if $\delta>0$.}
	\end{aligned}
	\right.
\end{align*}
Due to the assumptions on the initial data and recalling that $B^+_\nu(z), B^-_\nu(z)>0$ when $\nu>-1$ and $z>0$, we can observe that $c_+>0$ (see also Remark~\ref{rem:weaker-initial-data}).
Exploiting the asymptotic expansions for the modified Bessel's functions (equations (9.7.1) and (9.7.2) from Chapter~9 in~\cite{AS}), we have that
\begin{equation*}
	\B_0(z) = \e \left[ c_+ \frac{e^z}{\sqrt{2\pi z}} + c_- \sqrt{\frac{\pi}{2z}} e^{-z} \right] (1+O(1/z)).
\end{equation*}
Then, there exist two constants $C>0$ and $z_0 \ge 1$, both not depending on $\e$, such that
\begin{equation}
\label{estB0}
	\B_0(z) \ge C\e z^{-1/2} e^z \quad\text{for $z\ge z_0$.}
\end{equation}

\textit{Inhomogeneous problem.} Let us consider now the Cauchy problem
\begin{equation*}
\left\{
\begin{aligned}
& z^2\frac{d^2\B_N}{dz^2} (z)
+ z \frac{d\B_N}{dz} (z)
- \left( z^2 +\frac{\delta}{4} \right) \B_N(z)
= N(z),
\quad\text{$z \ge 1$}, \\
& \B_N(1) =
\frac{d\B_N}{dz}(1) =
0.
\end{aligned}
\right.
\end{equation*}
Exploiting the method of variation of parameters, we have that
\begin{equation*}
	\B_N(z) = B^+_{\sqrt{\delta}/2}(z) \int_{1}^{z} \xi B^-_{\sqrt{\delta}/2}(\xi) N(\xi) d\xi
	- B^-_{\sqrt{\delta}/2}(z) \int_{1}^{z} \xi B^+_{\sqrt{\delta}/2}(\xi) N(\xi) d\xi.
\end{equation*}
Recalling that $N(z)\ge0$ and using the fact that $B^+_{\sqrt{\delta}/2}(z)$ is increasing and $B^-_{\sqrt{\delta}/2}(z)$ is decreasing respect to the argument for $z > 0$ (due to relations (9.6.26) from Chapter~9 in~\cite{AS}), we get that
\begin{equation}
\label{estBN}
	\B_N(z) \ge 0 \quad\text{for $z\ge1$.}
\end{equation}

Since the solution $\B$ to the Cauchy problem \eqref{kummer}-\eqref{kummer_data} is the sum of $\B_0$ and $\B_N$ and from estimates \eqref{estB0} and \eqref{estBN}, we get
\begin{equation*}
	\B(z) =
	\B_0(z) +\B_N(z)
	\gtrsim \e z^{-1/2} e^z \quad\text{for $z\ge z_0$}.
\end{equation*}
So, recalling the definition \eqref{def:B} of $\B$ and changing again the variables, we reach
\begin{equation}
\label{estF1}
	F_1(t) \gtrsim \e (1+t)^{-\mu_1/2} \quad\text{for $t \gtrsim 1$.}
\end{equation}
By H\"older inequality and using estimates \eqref{YZ} and \eqref{estF1}, we obtain
\begin{equation*}
\begin{split}
\int_{\R^n} |u(x,t)|^p dx
&\geq
\left( \int_{\R^n} |\psi_1(x,t)|^{p/(p-1)}\right)^{1-p} |F_1(t)|^p
\\
&\gtrsim \e^p
  (1+t)^{-(n+\mu_1-1)\frac{p}{2}+n-1}
  \quad\text{for $t\gtrsim 1 $},
\end{split}
\end{equation*}
which inserted in \eqref{4} and recalling that $L(t)$ is positive for $t$ great enough, give us
\begin{equation*}
\begin{split}
	%
	F_0(t) & \gtrsim \e^p
	(1+t)^{-\k} \int_{T_1}^{t}
	(1+s)^{-\l} ds \int_{T_1}^{s}
	(1+r)^{q+\sqrt{\delta}-1}dr
	\quad\text{for $t \ge T_1 $},
	\end{split}
\end{equation*}
for a suitable $T_1>0$, where we define
\begin{equation}
\label{def:q}
	q \equiv q(p)
	:= \k-(n+\mu_1-1)\frac{p}{2}+n+1.
\end{equation}
We obtain, for large time $t\gtrsim1$, that:
\begin{itemize}
	\item if $q>-\sqrt{\delta}$, then
	\begin{equation*}
		F_0(t) \gtrsim \e^p t^{-\k}
		\left\{
		\begin{aligned}
		& t^q
		&\quad&\text{if $q>0$,}
		\\
		& \ln(1+t)
		&\quad&\text{if $q=0$,}
		\\
		& 1
		&\quad&\text{if $q<0$;}
		\end{aligned}
		\right.
	\end{equation*}
	
	\item if $q=-\sqrt{\delta}$, then
		\begin{equation*}
			F_0(t) \gtrsim \e^p t^{-\k}
			\left\{
			\begin{aligned}
			& 1
			&\quad&\text{if $\delta>0$,}
			\\
			& \ln^2(1+t)
			&\quad&\text{if $\delta=0$;}
			\end{aligned}
			\right.
		\end{equation*}
	
	\item if $q<-\sqrt{\delta}$, then
		\begin{equation*}
			F_0(t) \gtrsim \e^p t^{-\k}
			\left\{
			\begin{aligned}
			& 1
			&\quad&\text{if $\delta>0$,}
			\\
			& \ln(1+t)
			&\quad&\text{if $\delta=0$.}
			\end{aligned}
			\right.
		\end{equation*}
\end{itemize}
Summing all up, we deduce the relations in \eqref{key2}.


\subsection{Application of Kato's type lemma}\label{sec:kato_appl}
Now we will proceed applying the Kato's type lemma, as presented in Section \ref{sec:kato}, two times to two different couples of inequalities, and subsequently we will infer which result is optimal. The calculations of this subsection are all elementary (and quite tedious), so we will only sketch them.

Apply Lemma~\ref{kato} to the inequalities \eqref{key1} and \eqref{key3}, with
\begin{gather*}
	E=\e,
	\\
	a = \left\{
	\begin{aligned}
	& 1
	&\quad&\text{if $H_0=0$,}
	\\
	& \l
	&\quad&\text{if $H_0>0$.}
	\end{aligned}
	\right.
	\qquad
	b=(n+\k+\l)(p-1),
	\\
	c = \left\{
	\begin{aligned}
	& 0
	&\quad&\text{if $H_0=0$,}
	\\
	& 1-\sgn\delta
	&\quad&\text{if $H_0>0$,}
	\end{aligned}
	\right.
	\\
	1< p < p_c :=
	\left\{
	\begin{aligned}
	& p_F(n+\k+\sqrt{\delta})
	&\quad&\text{if $H_0=0$,}
	\\
	& p_F(n+\k)
	&\quad&\text{if $H_0>0$,}
	\end{aligned}
	\right.
	\\
	\gamma =
	\left\{
	\begin{aligned}
	& 2\gamma_F(p,n+\k+\sqrt{\delta})
	&\quad&\text{if $H_0=0$,}
	\\
	& 2\gamma_F(p,n+\k)
	&\quad&\text{if $H_0>0$.}
	\end{aligned}
	\right.
\end{gather*}
We chose $p \in (1,p_c)$ since this is equivalent to $\gamma>0$ for $p>1$.
Then, for every $p\in(1,p_c)$, we have $T_\e \lesssim \widetilde{T} \equiv \widetilde{T}(\e)$, with
\begin{equation}
\label{Ttil}
\e^p \widetilde{T}^{\frac{p\gamma}{p-1}}
\left[\ln(1+\widetilde{T})\right]^{pc}
=1.
\end{equation}

Apply Lemma~\ref{kato} to the inequalities \eqref{key1} and \eqref{key2}, with
\begin{gather*}
\overline{E}=\e^p,
\\
\overline{a} = \left\{
\begin{aligned}
& \l+q
&\quad&\text{if $q>0$,}
\\
& \l
&\quad&\text{if $q \le 0$,}
\end{aligned}
\right.
\qquad
\overline{b}=(n+\k+\l)(p-1),
\\
\overline{c} = \left\{
\begin{aligned}
& 0
&\quad&\text{if $q>0$,}
\\
& 2-\sgn\delta
&\quad&\text{if $q=0$,}
\\
& 1-\sgn\delta
&\quad&\text{if $q<0$,}
\end{aligned}
\right.
\\
1 < p < \overline{p}_c,
%
\qquad
\overline{\gamma} =
\left\{
\begin{aligned}
& \gamma_S(p,n+\mu_1)
&\quad&\text{if $q>0$,}
\\
& 2\gamma_F(p,n+\k)
&\quad&\text{if $q \le 0$,}
\end{aligned}
\right.
\end{gather*}
where $q$ is the one in \eqref{def:q} and $\overline{p}_c \in (1,+\infty]$ is defined as the exponent such that $\overline{\gamma} >0$ for $1<p<\overline{p}_c$ (we will explicitly define this exponent later).
Then, for every $p\in(1,\overline{p}_c)$, we have $T_\e \lesssim \widetilde{S} \equiv \widetilde{S}(\e)$, with
\begin{equation}
\label{Stil}
\e^p \widetilde{S}^{\frac{\overline{\gamma}}{p-1}}
\left[\ln(1+\widetilde{S})\right]^{\overline{c}}
=1.
\end{equation}
In both cases, since \eqref{key1}, \eqref{key2} and \eqref{key3} are true for $t\ge T_0$ with some time $T_0$, and since we need to require $\widetilde{T},\widetilde{S} \ge T_0$ to apply the Kato's type lemma, we need to impose also that $\e$ is sufficiently small.
From these computations, we deduce the blow-up for $1<p<p_{k} := \max\{p_c, \overline{p}_c\}$ and the upper lifespan estimate $T_\e \lesssim \min\{\widetilde{T}, \widetilde{S}\}$. We will go further in the analysis to clarify these values.

Before to move forward, in order to understand the definition of $\widetilde{S}$ we need to write down more explicitly the definitions of $\overline{c}$, $\overline{p}_c$ and $\overline{\gamma}$, since they depend on $q$ and therefore on the exponent $p$. Firstly, recall the definition \eqref{def:p*} of $p_*=p_*(n+\mu_1,n-\sqrt{\delta})$
and that, by \eqref{p*rel}, for $p>1$ and $\mu_1 +n \neq 1$, it holds
\begin{equation*}
	p=p_*
	\Longleftrightarrow
	q(p)=0
	\Longleftrightarrow
	\gamma_S(p,n+\mu_1) = 2\gamma_F(p,n+\k).
\end{equation*}
We will consider several cases, due to the generality of the constants involved, but what lies beneath is the elementary comparison between the parabola $\gamma_S$ (line in the case $\mu_1+n=1$) and the line $2\gamma_F$. Also, since we want to be in the hypothesis of Kato's type lemma, our interest is directed to $\overline{\gamma}>0$, and so we explicit its definition only for the range $1<p<\overline{p}_c$.

\textit{Case 1: $n+\mu_1 > 1$.} Recalling the definition \eqref{def:d*1}--\eqref{def:d*2} of $d_* := d_*(n+\mu_1)$ and the relation \eqref{d*rel}, we have that the following hold true:
\begin{gather*}
	0< d_* <2,
	\\
	\sqrt{\delta} = n-d_* \Longleftrightarrow
	p_* = p_S(n+\mu_1) = p_F(n+\k) = \frac{2}{d_*}.
\end{gather*}
Taking also in account that
\begin{align*}
	\sqrt{\delta} \le
	n-d_*(n+\mu_1)
	&\Longleftrightarrow
	p_* \ge p_S(n+\mu_1)
	,
	\\
	\sqrt{\delta} < n+2
	&\Longleftrightarrow
	p_*>1
	,
	\\
	q > 0 &\Longleftrightarrow p < p_*,
\end{align*}
we have:
\begin{itemize}
	\item if $\sqrt{\delta} \le
	n-d_*$, then
	\begin{gather*}
	\overline{p}_c = p_S(n+\mu_1),
	\\
	\overline{\gamma} = \gamma_S(p,n+\mu_1),
	\quad\text{for $1<p<\overline{p}_c$,}
	\\
	\overline{c}=0;
	\end{gather*}
	
	\item if $n-d_* < \sqrt{\delta} < n+2$, then
	\begin{gather*}
	\overline{p}_c = p_F(n+\k),
	\\
	\overline{\gamma} =
	\left\{
	\begin{aligned}
	&\gamma_S(p,n+\mu_1), &\quad&\text{for $1<p< p_*$,}
	\\
	&2\gamma_F(p,n+\k), &\quad&\text{for $p_* \le p<\overline{p}_c$,}
	\end{aligned}
	\right.
	\\
	\overline{c} =
	\left\{
	\begin{aligned}
	&0, &\quad&\text{for $1<p< p_*$,}
	\\
	&2-\sgn\delta, &\quad&\text{for $p=p_*$,}
	\\
	&1-\sgn\delta, &\quad&\text{for $p_* < p<\overline{p}_c$;}
	\end{aligned}
	\right.
	\end{gather*}
	
	\item if $\sqrt{\delta} \ge
	n+2$, then
	\begin{gather*}
	\overline{p}_c = p_F(n+\k),
	\\
	\overline{\gamma} = 2\gamma_F(p,n+\k)
	\quad\text{for $1<p<\overline{p}_c$,}
	\\
	\overline{c} = 1-\sgn\delta.
	\end{gather*}
\end{itemize}

\textit{Case 2: $n+\mu_1=1$.} Taking in account that
\begin{equation*}
	q>0 \Longleftrightarrow \sqrt{\delta}<n+2
\end{equation*}
we have:
\begin{itemize}
	\item if $\sqrt{\delta}<n+2$, then
	\begin{gather*}
		\overline{p}_c = p_S(n+\mu_1) = p_S(1) = +\infty,
		\\
		\overline{\gamma} = \gamma_S(p,n+\mu_1) = \gamma_S(p,1) = 2+2p,
		\quad\text{for $1<p<\overline{p}_c$},
		\\
		\overline{c}=0;		
	\end{gather*}
	
	\item if $\sqrt{\delta}=n+2$, then
	\begin{gather*}
	\overline{p}_c = p_S(n+\mu_1) = p_F(n+\k) = +\infty,
	\\
	\overline{\gamma} = \gamma_S(p,n+\mu_1) = 2\gamma_F(p,n+\k) = 2+2p,
	\quad\text{for $1<p<\overline{p}_c$},
	\\
	\overline{c}=2-\sgn\delta;		
	\end{gather*}
	
	\item if $\sqrt{\delta} > n+2$, then
	\begin{gather*}
	\overline{p}_c = p_F(n+\k) = p_F\left((n-\sqrt{\delta})/2\right) = +\infty,
	\\
	\overline{\gamma} = 2\gamma_F(p,n+\k) = 2\gamma_F \left(p,(n-\sqrt{\delta})/2 \right),
	\quad\text{for $1<p<\overline{p}_c$},
	\\
	\overline{c} = 1-\sgn\delta.
	\end{gather*}
\end{itemize}

\textit{Case 3: $n+\mu_1<1$.}
Taking in account that
\begin{align*}
p_*>1 &\Longleftrightarrow \sqrt{\delta} > n+2,
\\
q > 0 &\Longleftrightarrow p > p_*,
\end{align*}
we have:
\begin{itemize}
	\item if $\sqrt{\delta} \le
	n+2$, then
	\begin{gather*}
	\overline{p}_c = p_S(n+\mu_1) = +\infty,
	\\
	\overline{\gamma} = \gamma_S(p,n+\mu_1)
	\quad\text{for $1<p<\overline{p}_c$,}
	\\
	\overline{c}=0;
	\end{gather*}
	
	\item if $\sqrt{\delta} > n+2$, then
	\begin{gather*}
	\overline{p}_c = p_S(n+\mu_1) = +\infty,
	\\
	\overline{\gamma} =
	\left\{
	\begin{aligned}
	&2\gamma_F(p,n+\k), &\quad&\text{for $1<p \le p_*$,}
	\\
	&\gamma_S(p,n+\mu_1), &\quad&\text{for $p_* < p<\overline{p}_c$,}
	\end{aligned}
	\right.
	\\
	\overline{c} =
	\left\{
	\begin{aligned}
	&1-\sgn\delta, &\quad&\text{for $1<p< p_*$,}
	\\
	&2-\sgn\delta, &\quad&\text{for $p=p_*$,}
	\\
	&0, &\quad&\text{for $p_* < p<\overline{p}_c$.}
	\end{aligned}
	\right.
	\end{gather*}
\end{itemize}

Now that the definitions of $p_c$, $\overline{p}_c$ and $\widetilde{T}$, $\widetilde{S}$ are clear, we can go further.

\subsection{Comparison between the obtained exponents and lifespans}\label{sub:comparison}

As we said, from our computations we found the blow-up for $1 < p < p_{k} = \max\{p_c, \overline{p}_c\}$
and the upper lifespan estimates  $T_\e \lesssim \min\{\widetilde{T}, \widetilde{S}\}$. Observing that $\widetilde{T}(\e), \widetilde{S}(\e) \to +\infty$ for $\e \to 0$ and comparing the relations \eqref{Ttil} and \eqref{Stil}, we get that $\widetilde{T} \lessgtr \widetilde{S}$ if $p\gamma \gtrless \overline{\gamma}$. If $p\gamma = \overline{\gamma}$, the exponent of the logarithm comes into play, indeed $\widetilde{T} \lesseqgtr \widetilde{S}$ if $pc \gtreqless \overline{c}$.
Now, we need to consider two cases according to the fact that $H_0 = \ints h(x)dx$ is positive or null.

\textit{Case $H_0 > 0$.}
We can easily infer that $p_k = p_{\mu_1,\delta}(n)$ defined in \eqref{pch>0}. We establish the upper bound for the lifespan $T_{\e}$ without making distinctions according to the value of $n+\mu_1$. Taking in account that, for $p>1$,
\begin{align*}
	2 p \, \gamma_F(p,n+\k) > \gamma_S(p,n+\mu_1)
	\Longleftrightarrow&\,
	\left\{
	\begin{aligned}
	&p>1, &\quad&\text{if $\sqrt{\delta}\ge n$,}
	\\
	&1<p<\frac{2}{n-\sqrt{\delta}},
	&\quad&\text{if $n-2<\sqrt{\delta}<n$,}
	\end{aligned}
	\right.
	\\
	\text{$n-d_*<\sqrt{\delta}<n$ and $n+\mu_1>1$}
	\Longrightarrow&\,
	p_F(n+\k) < \frac{2}{n-\sqrt{\delta}},
	\\
	%
	\text{$\sqrt{\delta} \le n-d_*$
	and
	$1<p<p_{k}$}
	\Longrightarrow&\,
	q>0,
\end{align*}
we have:
\begin{itemize}
	\item if $\sqrt{\delta}\le n-2$ and $1<p<p_{k}$, then
	$p\gamma<\overline{\gamma}$ and so $\widetilde{S} < \widetilde{T}$;
	
	\item if $n-2<\sqrt{\delta}<n-d_*$ and
		\begin{itemize}[label=$\circ$]
		\item if $1<p<\frac{2}{n-\sqrt{\delta}}$,
			then $p\gamma>\overline{\gamma}$ and so $\widetilde{T} < \widetilde{S}$;
		\\
		\item if $p=\frac{2}{n-\sqrt{\delta}}$,
		then
		$p\gamma=\overline{\gamma}$
		and
		$pc \ge \overline{c}$, so that $\widetilde{T} \le \widetilde{S}$;
		\\
		\item if $\frac{2}{n-\sqrt{\delta}}<p<p_{k}$, then $p\gamma<\overline{\gamma}$, so that $\widetilde{S} < \widetilde{T}$;
		\end{itemize}
		
	\item if $\sqrt{\delta} \ge n-d_*$ and if $1<p<p_{k}$, then
	$p\gamma>\overline{\gamma}$
	so that
	$\widetilde{T} < \widetilde{S}$.
\end{itemize}

\textit{Case $H_0=0$.} From now on we will impose the additional hypothesis that $\mu_1 > 0$ (which can be relaxed to $n+\mu_1 > 1$).

Obviously, $p_F(n+\k+\sqrt{\delta}) \le p_F(n+\k)$, hence again $p_k = p_{\mu_1,\delta}(n)$ defined in \eqref{pch>0}.
Consider that, for $p>1$,
\begin{align*}
	p\gamma_F(p,n+\k+\sqrt{\delta}) > \gamma_F(p,n+\k)
	&\Longleftrightarrow
	\text{$\sqrt{\delta}<2$
	and
	$1<p< 1+\frac{2-\sqrt{\delta}}{n+\k+\sqrt{\delta}}$}
	;
	\\
	2p \, \gamma_F(p,n+\k+\sqrt{\delta}) > \gamma_S(p,n+\mu_1)
	&\Longleftrightarrow
	\text{$n=1$
	and
	$\sqrt{\delta}<1$
	and
	$1<p<\frac{2}{1+\sqrt{\delta}}$.}
\end{align*}
If $n\ge2$, taking in account that
\begin{equation*}
	n-d_*<\sqrt{\delta}<n+2
	\Longrightarrow
	1+\frac{2-\sqrt{\delta}}{n+\k+\sqrt{\delta}} < p_*,
\end{equation*}
we can prove that $p\gamma<\overline{\gamma}$ for $1<p<p_{k}$, and so $\widetilde{S} < \widetilde{T}$.

Suppose now that $n=1$. Recall the definition \eqref{theta} of $\theta$ and note that it satisfies $\sgn\theta = \sgn\{\mu_1-3\}$, and moreover that the following relations hold:
\begin{align*}
	\mu_1 >0 &\Longrightarrow
	\text{$1-d_* < 1$
	and
	$1+\frac{2-\sqrt{\delta}}{n+\k+\sqrt{\delta}} < p_S(1+\mu_1)$,}
	\\
	0<\mu_1<3
	&\Longleftrightarrow
	1-d_* > 0,
	\\
	0<\mu_1<3 &\Longrightarrow |1-d_*|> \theta,
	\\
	\sqrt{\delta} > -1 + d_*
	&\Longrightarrow
	\frac{2}{1+\sqrt{\delta}} < p_S(1+\mu_1),
	\\
	\theta < \sqrt{\delta} < 3
	&\Longrightarrow
	\text{$1+\frac{2-\sqrt{\delta}}{n+\k+\sqrt{\delta}}<p_*$ and $\frac{2}{1+\sqrt{\delta}} < p_*$}.
\end{align*}
Recall also the definition \eqref{r_*} of $r_* \equiv r_*(\mu_1,\delta)$ and Remark\til\ref{rem:r*}.
%
%
Hence, we get that:
\begin{itemize}
	\item if $\sqrt{\delta}=0$, $\mu_1 =3$ and if $1<p<p_{k}$,
	then $p\gamma>\overline{\gamma}$ and so $\widetilde{T} < \widetilde{S}$;
	
	\item if $\sqrt{\delta}=0$ and $\mu_1 \neq 3$, or if $0<\sqrt{\delta} < 1$, we have:
	\begin{itemize}[label=$\circ$]
		\item if $1<p<r_*$,
		then $p\gamma>\overline{\gamma}$ and so $\widetilde{T} < \widetilde{S}$;
		\\
		\item if $p=r_*$,
		then
		$p\gamma=\overline{\gamma}$
		and
		$pc \le \overline{c}$, so that $\widetilde{S} \le \widetilde{T}$;
		\\
		\item if $r_*<p<p_{k}$, then $p\gamma<\overline{\gamma}$, so that $\widetilde{S} < \widetilde{T}$;
	\end{itemize}
	
	\item if $\sqrt{\delta} \ge 1$ and if $1<p<p_{k}$, then
	$p\gamma < \overline{\gamma}$
	so that
	$\widetilde{S} < \widetilde{T}$.
\end{itemize}

At the end, recalling the definitions of $\gamma, \overline{\gamma}, c$ and $\overline{c}$ in the various cases and summing all up, we conclude the proof for Theorem \ref{thm1} and Theorem \ref{thm:H=0}.

\subsection{Proof for Theorem\til\ref{thm_negmass}}\label{sub:negmass}

We will only sketch the demonstration, since it is a variation of the previous one. Let us introduce the functional
\begin{equation*}
	G_0(t) = \ints w(x,t) dx
\end{equation*}
and, as in \cite{LST19.1, LST19.2}, the bounded multiplier
\begin{equation*}
	m(t) := \exp\left( \nu_1 \frac{(1+t)^{1-\beta}}{1-\beta}\right) .
\end{equation*}
Choosing the test function $\phi=\phi(x,s)$ in \eqref{def:negmass} to satisfy
\eqref{eq:testfunction}, deriving respect to the time and multiplying by $m$, we get that

\begin{equation*}
	[m(t)G_0'(t)]' + \frac{\nu_2}{(1+t)^2} m(t)G_0(t) = m(t) \ints |w(x,t)|^p dx,
\end{equation*}
and hence
\begin{equation}\label{G0}
	\begin{split}
		G_0(t) = &\ G_0(0) + m(0)G_0'(0) \int_{0}^{t}m^{-1}(s)ds
		\\
		&- \int_{0}^{t} m^{-1}(s)ds \int_0^s m(r) \frac{\nu_2}{(1+r)^2} G_0(r) dr
		\\
		&+ \int_{0}^{t} m^{-1}(s)ds \int_0^s m(r) dr \ints |w(x,r)|^p dx.
	\end{split}
\end{equation}
It is simple to see, by comparison argument, that $G_0$ is positive. Indeed, by the hypothesis on initial data, we know that $G_0(0) = \ints f(x)dx$ and $G'_0(0) = \ints g(x)dx$ are non-negative and not both zero. If $G_0(0)>0$, by continuity $G_0$ is positive for small time. If $G_0(0)=0$ and $G'(0)>0$, then $G_0$ is increasing and again positive for small time $t>0$. If we suppose that there exist a time $t_0>0$ such that $G_0(t_0)=0$, calculating \eqref{G0} in $t=t_0$ we get a contradiction, since the left-hand term would be zero and the right-hand term would be strictly positive. Then, $G_0$ is positive for any time $t>0$.
Define now the functional $\overline{G}_0$ as the solution of the integral equation
\begin{equation}\label{GG0}
\begin{split}
\o{G}_0(t) = &\ \frac{1}{2}G_0(0) + \frac{m(0)}{2}G_0'(0) t
- m(0) \int_{0}^{t} ds \int_0^s \frac{\nu_2}{(1+r)^2} \o{G}_0(r) dr
\\
&+ m(0) \int_{0}^{t} ds \int_0^s dr \ints |w(x,r)|^p dx.
\end{split}
\end{equation}
Since $m(0) < m(t) <1$ for any $t>0$ and $\nu_2 < 0$, we have that
\begin{equation*}
G_0(t) - \o{G}_0(t) \ge \frac{1}{2}G_0(0) + \frac{m(0)}{2}G_0'(0) t
- m(0) \int_{0}^{t} ds \int_0^s \frac{\nu_2}{(1+r)^2} [G_0(r) - \o{G}_0(r)] dr,
\end{equation*}
and, again by comparison argument, we infer that $G_0 \ge \o{G}_0$. From \eqref{GG0} we get that $\o{G}_0$ satisfies
\begin{equation*}
	\o{G}''_0(t) + \frac{m(0)\nu_2}{(1+t)^2} \o{G}(t) = m(0)\ints |w(x,t)|^p dx,
\end{equation*}
which has the same structure of \eqref{2} with $\mu_1=0$ and $\mu_2=m(0)\nu_2$. Setting
\begin{equation*}
	{\lambda} := 1 + \sqrt{{\delta}},
	\quad
	{\k} := -{\l}/2,
	\quad
	\mathcal{G}(t) := (1+t)^{\k+\l} \overline{G}_0(t),
\end{equation*}
similarly as in Subsection~\ref{sec:F0} we obtain
\begin{equation}\label{overG}
\begin{split}
\o{G}_0(t) = &\ \o{G}_0(0)(1+t)^{-\k} + [\k \o{G}_0(0) + \o{G}'_0(0)] (1+t)^{-\k} \int_{0}^{t}(1+s)^{-\l} ds \\
&+ (1+t)^{-\k}
\int_{0}^{t}(1+s)^{-\l} ds
\int_{0}^{s}(1+r)^{\k+\l}dr
\int_{\R^n}|w(x,r)|^p dx
\end{split}
\end{equation}
and then
\begin{align}
\label{Gkey1}
\mathcal{G}(t) &\gtrsim \int_{T_0}^{t}ds \int_{T_0}^{s} r^{-(n+\k+\l)(p-1)} \mathcal{G}(r)^p dr,
\\
\label{Gkey3}
\mathcal{G}(t) &\gtrsim \e t^{\l}.
\end{align}

Now, to get the counterpart of \eqref{key2}, define the functional
\begin{equation*}
G_1(t):=\int_{\R^n}w(x,t)\psi_1(x,t)dx,
\end{equation*}
with $\psi_1$ defined in \eqref{YZtest}. After a derivation respect to the time of the definition of energy solution \eqref{def:negmass} and multiplying both of its sides with $m(t)$, we have that
\[
\begin{split}
&\frac{d}{dt}\left\{m(t)
\int_{\R^n}w_t(x,t)\phi(x,t)dx\right\}\\
&+m(t)\int_{\R^n}\left\{-w_t(x,t)\phi_t(x,t)-w(x,t)\Delta\phi(x,t)\right\}dx\\
=&\ - m(t)\int_{\R^n} \frac{\nu_2}{(1+t)^{2}}w(x,t)\phi(x,t)dx + m(t)\int_{\R^n}|w(x,t)|^p\phi(x,t)dx.
\end{split}
\]
By integration on $[0,t]$ we get
\begin{equation*}
\begin{split}
&m(t)
\int_{\R^n}w_t(x,t)\phi(x,t)dx
-m(0)\e\int_{\R^n}g(x)\phi(x,0)dx\\
&-m(t)\int_{\R^n} w(x,t)\phi_t(x,t)dx + m(0)\e \int_{\R^n} f(x)\phi_t(x,0)dx\\
&+ \int_0^t ds \int_{\R^n} m(s)\frac{\nu_1}{(1+s)^\beta}w(x,s)\phi_t(x,s)dx
\\
& + \int_0^t ds \int_{\R^n} m(s)w(x,s)\phi_{tt}(x,s)dx - \int_0^tds\int_{\R^n} m(s)
w(x,s)\Delta\phi(x,s) \\
=&\ - \int_0^tds\int_{\R^n}m(s) \frac{\nu_2}{(1+s)^{2}}w(x,s)\phi(x,s)dx\\
&+ \int_0^tds \int_{\R^n} m(s)|w(x,s)|^p\phi(x,s)dx.
\end{split}
\end{equation*}
Setting
$
\phi(x,t)=\psi_1(x,t)=e^{-t}\phi_1(x)
$
on $\supp w$ and recalling the bounds on the multiplier $m(t)$,
we obtain
\[
\begin{split}
G_1'(t)+2G_1(t)
\ge &\ m(0) G'_1(0) + 2m(0)G_1(0) \\
&+ m(0) \int_0^t
\left\{\frac{\nu_1}{(1+s)^\beta} - \frac{\nu_2}{(1+s)^{2}} \right\} G_1(s)ds\\
&+ m(0) \int_0^tds\int_{\R^n} |w(x,s)|^pdx.
\end{split}
\]
Integrating the above inequality over $[0,t]$ after a multiplication by $e^{2t}$, we get
\begin{equation*}\label{G}
\begin{split}
G_1(t)
\ge &\
G_1(0) e^{-2t} +
m(0) \{ G'_1(0) + 2 G_1(0) \} \frac{1-e^{-2t}}{2}
\\
& +m(0) e^{-2t} \int_0^t e^{2s} ds \int_0^s
\left\{\frac{\nu_1}{(1+r)^\beta} - \frac{\nu_2}{(1+r)^{2}} \right\} G_1(r)dr
\\
& +m(0) e^{-2t} \int_0^t e^{2s} ds \int_0^s
dr \int_{\R^n} |w(x,r)|^p\phi(x,r)dx,
\end{split}
\end{equation*}
from which, thanks again to a comparison argument, we infer that $G_1$ is non-negative, and so, neglecting the last two term in the above inequality, it is easy to reach
\begin{equation*}
	G_1(t) \gtrsim \e
	\quad
	\text{for $t \gtrsim 1$.}
\end{equation*}
Hence, we have also
\begin{equation*}
	\int_{\R^n} |w(x,t)|^p dx \gtrsim \e^p (1+t)^{-(n-1)\frac{p}{2} +n -1}
	\quad
	\text{for $t \gtrsim 1$,}
\end{equation*}
and so, taking in account \eqref{overG}, it holds
\begin{equation*}
\begin{split}
\overline{G}_0(t) & \gtrsim \e^p
(1+t)^{-\k} \int_{T_1}^{t}
(1+s)^{-\l} ds \int_{T_1}^{s}
(1+r)^{q+\sqrt{\delta}-1}dr
\quad\text{for $t \ge T_1 $},
\end{split}
\end{equation*}
for some $T_1>0$, where
\begin{equation*}
q \equiv q(p)
:= -\frac{1+\sqrt{\delta}}{2}-(n-1)\frac{p}{2}+n+1.
\end{equation*}
Finally, we obtain the inequality analogous to \eqref{key2}, i.e.
\begin{equation}\label{Gkey2}
\mathcal{G}(t) \gtrsim \e^p
\left\{
\begin{aligned}
& t^{\l+q}
&\quad&\text{if $q>0$,}
\\
& t^{\l} \ln(1+t)
&\quad&\text{if $q=0$,}
\\
& t^{\l}
&\quad&\text{if $q<0$.}
\end{aligned}
\right.
\end{equation}
Thanks to \eqref{Gkey1}, \eqref{Gkey3} and \eqref{Gkey2} and applying the Kato's type lemma as in Subsection \ref{sec:kato_appl}, we can conclude the proof of Theorem \ref{thm_negmass}.


\section*{Acknowledgement}
The first author is supported by NSF of Zhejiang Province (LY18A010008) and NSFC (11771194).
The second author is a member of the Gruppo Nazionale per l'Analisi Matematica, la
Probabilit\`a e le loro Applicazioni (GNAMPA) of the Istituto Nazionale di
Alta Matematica (INdAM).
The third author is partially supported by the Grant-in-Aid for Scientific Research (B) (No.18H01132), Japan Society for the Promotion of Science.


\bibliographystyle{plain}

\begin{thebibliography}{99}
	
	\bibitem{AS}
	Abramowitz, M. and Stegun, I. A. (1965). Handbook of mathematical functions with formulas, graphs, and mathematical table. US Department of Commerce. National Bureau of Standards Applied Mathematics series 55.
	
	\bibitem{Dab15}
	D'Abbicco, M. (2015). The threshold of effective damping for semilinear wave equations. Mathematical Methods in the Applied Sciences, 38(6), 1032-1045.
	
	\bibitem{DL13}
	D'Abbicco, M. and Lucente, S. (2013). A modified test function method for damped wave equations. Advanced Nonlinear Studies, 13(4), 867-892.
	
	\bibitem{DL15}
	D'Abbicco, M. and Lucente, S. (2015). NLWE with a special scale invariant damping in odd space dimension. Discrete Contin. Dyn. Syst., Dynamical systems, differential equations and applications.
10th AIMS Conference. Suppl., 312-319.
	
	\bibitem{DLR13}
	D'Abbicco, M., Lucente, S. and Reissig, M. (2013). Semi-linear wave equations with effective damping. Chinese Annals of Mathematics, Series B, 34(3), 345-380.
	
	\bibitem{DLR15}
	D'Abbicco, M., Lucente, S. and Reissig, M. (2015). A shift in the Strauss exponent for semilinear wave equations with a not effective damping. Journal of Differential Equations, 259(10), 5040-5073.
	
	\bibitem{ER}
	Ebert, M. R. and Reissig, M. (2018). Methods for partial differential equations. Qualitative Properties of Solutions, Phase Space Analysis, Semilinear Models.
	
	\bibitem{FIW19}
	Fujiwara, K., Ikeda, M. and Wakasugi, Y. (2019). Estimates of lifespan and blow-up rates for the wave equation with a time-dependent damping and a power-type nonlinearity. Funkcialaj Ekvacioj, 62(2), 157-189.
	
	\bibitem{II19}
	Ikeda, M. and Inui, T. (2019). The sharp estimate of the lifespan for semilinear wave equation with time-dependent damping. Differential and Integral Equations, 32(1/2), 1-36.
	
	\bibitem{IO16}
	Ikeda, M. and Ogawa, T. (2016). Lifespan of solutions to the damped wave equation with a critical nonlinearity. Journal of Differential Equations, 261(3), 1880-1903.
	
	\bibitem{IS18}
	Ikeda, M. and Sobajima, M. (2018). Life-span of solutions to semilinear wave equation with time-dependent critical damping for specially localized initial data. Mathematische Annalen, 372(3-4), 1017-1040.
	
	\bibitem{IW15}
	Ikeda, M. and Wakasugi, Y. (2015). A note on the lifespan of solutions to the semilinear damped wave equation. Proceedings of the American Mathematical Society, 143(1), 163-171.
	
	\bibitem{IW20}
	Ikeda, M. and Wakasugi, Y. (2020). Global well-posedness for the semilinear wave equation with time dependent damping in the overdamping case. Proceedings of the American Mathematical Society, 148(1), 157-172.
	
	\bibitem{IKTW19-aa}
	Imai, T., Kato, M., Takamura, H. and Wakasa, K. (2019). The sharp lower bound of the lifespan of solutions to semilinear wave equations with low powers in two space dimensions. In Asymptotic Analysis for Nonlinear Dispersive and Wave Equations, 81, 31-53. Mathematical Society of Japan.
	
	\bibitem{IKTW19-arxiv}
	Imai, T., Kato, M., Takamura, H. and Wakasa, K. (2019). The lifespan of solutions of semilinear wave equations with the scale-invariant damping in two space dimensions. arXiv preprint arXiv:1910.11692.
	
	\bibitem{KS19}
	Kato, M. and Sakuraba, M. (2019). Global existence and blow-up for semilinear damped wave equations in three space dimensions. Nonlinear Analysis, 182, 209-225.
	
	\bibitem{KTW19}
	Kato, M., Takamura, H. and Wakasa, K. (2019). The lifespan of solutions of semilinear wave equations with the scale-invariant damping in one space dimension. Differential and Integral Equations, 32(11/12), 659-678.
	
	\bibitem{KQ01}
	Kirane, M. and Qafsaoui, M. (2002). Fujita's exponent for a semilinear wave equation with linear damping. Advanced Nonlinear Studies, 2(1), 41-49.
	
	\bibitem{LaiN}
	Lai, N. A. (2018). Weighted $ L^2-L^2$ estimate for wave equation and its applications. arXiv preprint arXiv:1807.05109.
	
	\bibitem{LST19.1}
	Lai, N. A., Schiavone, N. M. and Takamura, H. (2019). Wave-like blow-up for semilinear wave equations with scattering damping and negative mass term. In New Tools for Nonlinear PDEs and Application (pp. 217-240). Birkh\"auser, Cham.
	
	\bibitem{LST19.2}
	Lai, N. A., Schiavone, N. M., and Takamura, H. (2019). Short time blow-up by negative mass term for semilinear wave equations with small data and scattering damping. arXiv preprint arXiv:1905.08100.
	
	\bibitem{LT18}
	Lai, N. A. and Takamura, H. (2018). Blow-up for semilinear damped wave equations with subcritical exponent in the scattering case. Nonlinear Analysis, 168, 222-237.
	
	\bibitem{LTW17}
	Lai, N. A., Takamura, H. and Wakasa, K. (2017). Blow-up for semilinear wave equations with the scale invariant damping and super-Fujita exponent. Journal of Differential Equations, 263(9), 5377-5394.
	
	\bibitem{LZ14}
	Lai, N. A. and Zhou, Y. (2014). An elementary proof of Strauss conjecture. Journal of Functional Analysis, 267(5), 1364-1381.
	
	\bibitem{LaiZ19}
	Lai, N. A. and Zhou, Y. (2019). The sharp lifespan estimate for semilinear damped wave equation with Fujita critical power in higher dimensions. Journal de Math'ematiques Pures et Appliqu\'ees, 123, 229-243.
			
	\bibitem{LZ95}
	Li, T. and Zhou, Y. (1995). Breakdown of solutions to $\square u+ u_t=| u|^{1+\alpha} $. Discrete \& Continuous Dynamical Systems-A, 1(4), 503.
	
	\bibitem{LNZ12}
	Lin, J., Nishihara, K. and Zhai, J. (2012). Critical exponent for the semilinear wave equation with time-dependent damping. Discrete Contin. Dyn. Syst, 32(12), 4307-4320.
	
	\bibitem{Lin90}
	Lindblad, H. (1990). Blow-up for solutions of $\square u = |u|^p$ with small initial data. Communications in partial differential equations, 15(6), 757-821.
	
	\bibitem{LS96}
	Lindblad, H. and Sogge, C. D. (1996). Long-time existence for small amplitude semilinear wave equations. American Journal of Mathematics, 118(5), 1047-1135.

\bibitem{LW}Liu, M. Y. and Wang, C. B. (2020).
Global existence for semilinear damped wave equations in
relation with the Strauss conjecture. Discrete Contin. Dyn. Syst, 40(2), 709-724.
	
	\bibitem{Nis11}
	Nishihara, K. (2011). Asymptotic behavior of solutions to the semilinear wave equation with time-dependent damping. Tokyo Journal of Mathematics, 34(2), 327-343.
	
	\bibitem{NPR17}
	Nunes do Nascimento, W., Palmieri, A. and Reissig, M. (2017). Semi-linear wave models with power non-linearity and scale-invariant time-dependent mass and dissipation. Mathematische Nachrichten, 290(11-12), 1779-1805.
	
	\bibitem{Pal18-thesis}
	Palmieri, A. (2018). Global in time existence and blow-up results for a semilinear wave equation with scale-invariant damping and mass (Doctoral dissertation, Technische Universit\"at Bergakademie Freiberg).
	
	\bibitem{Pal18}
	Palmieri, A. (2018). Global existence of solutions for semi-linear wave equation with scale-invariant damping and mass in exponentially weighted spaces. Journal of Mathematical Analysis and Applications, 461(2), 1215-1240.
	
	\bibitem{Pal19-mmas}
	Palmieri, A. (2019). A global existence result for a semilinear scale-invariant wave equation in even dimension. Mathematical Methods in the Applied Sciences, 42(8), 2680-2706.
	
	\bibitem{Pal19-newtools}
	Palmieri, A. (2019). Global existence results for a semilinear wave equation with scale-invariant damping and mass in odd space dimension. In New Tools for Nonlinear PDEs and Application (pp. 305-369). Birkh\"user, Cham.
	
	\bibitem{PR18}
	Palmieri, A. and Reissig, M. (2018). Semi-linear wave models with power non-linearity and scale-invariant time-dependent mass and dissipation, II. Mathematische Nachrichten, 291(11-12), 1859-1892.
	
	\bibitem{PR19-jde}
	Palmieri, A. and Reissig, M. (2019). A competition between Fujita and Strauss type exponents for blow-up of semi-linear wave equations with scale-invariant damping and mass. Journal of Differential Equations, 266(2-3), 1176-1220.
	
	\bibitem{PT19}
	Palmieri, A. and Tu, Z. (2019). Lifespan of semilinear wave equation with scale invariant dissipation and mass and sub-Strauss power nonlinearity. Journal of Mathematical Analysis and Applications, 470(1), 447-469.
	
	\bibitem{T15}
	Takamura, H. (2015). Improved Kato's lemma on ordinary differential inequality and its application to semilinear wave equations. Nonlinear Analysis, 125, 227-240.
	
	\bibitem{TW11}
	Takamura, H. and Wakasa, K. (2011). The sharp upper bound of the lifespan of solutions to critical semilinear wave equations in high dimensions. Journal of Differential Equations, 251(4-5), 1157-1171.
	
	\bibitem{TW14}
	Takamura, H. and Wakasa, K. (2014). Almost global solutions of semilinear wave equations with the critical exponent in high dimensions. Nonlinear Analysis: Theory, Methods \& Applications, 109, 187-229.
	
	\bibitem{TY01}
	Todorova, G. and Yordanov, B. T. (2001). Critical exponent for a nonlinear wave equation with damping. Journal of Differential equations, 174(2), 464-489.
		
	\bibitem{TL17}
	Tu, Z. and Lin, J. (2017). A note on the blowup of scale invariant damping wave equation with sub-Strauss exponent. arXiv preprint arXiv:1709.00866.
	
	\bibitem{TL19}
	Tu, Z. and Lin, J. (2019). Life-span of semilinear wave equations with scale-invariant damping: Critical Strauss exponent case. Differential and Integral Equations, 32(5/6), 249-264.
	
	\bibitem{wakasa16}
	Wakasa, K. (2016). The lifespan of solutions to semilinear damped wave equations in one space dimension. Communications on Pure \& Applied Analysis, 15(4).
	
	\bibitem{WY19}
	Wakasa, K. and Yordanov, B. T. (2019). On the nonexistence of global solutions for critical semilinear wave equations with damping in the scattering case. Nonlinear Analysis, 180, 67-74.
	
	\bibitem{wak14}
	Wakasugi, Y. (2014). Critical exponent for the semilinear wave equation with scale invariant damping. In Fourier analysis (pp. 375-390). Birkh\"auser, Cham.
	
	\bibitem{wak14-thesis}
	Wakasugi, Y. (2014). On the diffusive structure for the damped wave equation with variable coefficients (Doctoral dissertation, Doctoral thesis, Osaka University).
	
	\bibitem{wakasugi17}
	Wakasugi, Y. (2017). Scaling variables and asymptotic profiles for the semilinear damped wave equation with variable coefficients. Journal of Mathematical Analysis and Applications, 447(1), 452-487.
	
	\bibitem{Wir04}
	Wirth, J. (2004). Solution representations for a wave equation with weak dissipation. Mathematical methods in the applied sciences, 27(1), 101-124.
	
	\bibitem{Wir06}
	Wirth, J. (2006). Wave equations with time-dependent dissipation I. Non-effective dissipation. Journal of Differential Equations, 222(2), 487-514.
	
	\bibitem{Wir07}
	Wirth, J. (2007). Wave equations with time-dependent dissipation II. Effective dissipation. Journal of Differential Equations, 232(1), 74-103.
	
	\bibitem{YZ06}
	Yordanov, B. T. and Zhang, Q. S. (2006). Finite time blow up for critical wave equations in high dimensions. Journal of Functional Analysis, 231(2), 361-374.
	
	\bibitem{Zhang01}
	Zhang, Q. S. (2001). A blow-up result for a nonlinear wave equation with damping: the critical case. Comptes Rendus de l'Acad\'emie des Sciences-Series I-Mathematics, 333(2), 109-114.
	
	\bibitem{Zhou92-cam}
	Zhou, Y. (1992). Life span of classical solutions to $u_{tt}-u_{xx}=|u|^{1+\alpha}$. Chin. Ann. of Math. B, 13(2), 230-243.
	
	\bibitem{Zhou92-jde}
	Zhou, Y. (1992). Blow up of classical solutions to
	$\square u = |u|^{1+\alpha}$
	in three space dimensions. J. Partial Differential Equations, 5, 21-32.
	
	\bibitem{Zhou93}
	Zhou, Y. (1993). Life span of classical solutions to $\square u = |u|^p$ in two space dimensions. Chinese Ann. Math. Ser. B, 14(2), 225-236.
	
\end{thebibliography}

\end{document}